\documentclass[a4paper,12pt]{article}

\usepackage{amsthm,amsmath,stmaryrd,bbm,hyperref,geometry,color}
\usepackage[utf8]{inputenc}
\usepackage{amssymb}
\usepackage[english]{babel}
\usepackage{graphicx}
\usepackage{amsfonts,amssymb}
\usepackage{verbatim}
\usepackage{enumitem}

\newcommand{\po}{\left(}
\newcommand{\pf}{\right)}
\newcommand{\co}{\left[}
\newcommand{\cf}{\right]}
\newcommand{\cco}{\llbracket}
\newcommand{\ccf}{\rrbracket}
\newcommand{\R}{\mathbb R} 
\newcommand{\T}{\mathbb T} 
\newcommand{\Z}{\mathbb Z} 
 
\newcommand{\N}{\mathbb N} 
\newcommand{\dd}{\text{d}}
\newcommand{\na}{\nabla}
\newcommand{\1}{\mathbbm{1}}

\newtheorem{thm}{Theorem}

\newtheorem{prop}[thm]{Proposition}

\title{Kinetic walks for sampling}
\author{Pierre Monmarché}

\begin{document}
\maketitle


\abstract{The persistent walk is a classical model in kinetic theory, which has also been studied as a toy model for Markov Chain Monte Carlo questions. Its continuous limit, the telegraph process, has recently been extended to various velocity jump processes (Bouncy Particle Sampler, Zig-Zag process, etc.) in order to sample general target distributions on $\mathbb R^d$. This paper studies, from a sampling point of view, general kinetic walks that are natural discrete-time (and possibly discrete-space)  counterparts of these continuous-space processes. The main contributions of the paper are the definition and study of a discrete-space Zig-Zag sampler and the definition and time-discretisation of hybrid jump/diffusion kinetic samplers for multi-scale potentials on $\mathbb R^d$.}


\section{Introduction }

The classical persistent walk on $\Z$ is the Markov chain $(X_k,V_k)_{k\in\N}$ on $\Z\times\{-1,1\}$ with transitions
\[(X_{k+1},V_{k+1}) \ = \ \left\{ \begin{array}{ll}
(X_k + V_k,V_k) & \text{with probability }1-\alpha\\
(X_k,-V_k) & \text{else,}
\end{array}\right.\]
for some $\alpha\in[0,1]$. It describes the constant-speed motion of a self-propelled particle, $X_k$ denoting the position of the particle and $V_k$ its velocity. Since the time between two changes of the velocity follows a geometric distribution with parameter $\alpha$, $(\alpha X_{\lfloor \alpha t\rfloor},V_{\lfloor \alpha t\rfloor} )_{t\geqslant 0}$ naturally converges as $\alpha$ vanishes to the so-called telegraph process, for which the flips of the velocity are governed by a Poisson process \cite{HerrmannVallois}. From the seminal work of Goldstein \cite{Goldstein}, these two processes, and various extensions, have been studied in details, in particular from the point of view of statiscial physics and kinetic theory, or for other modelling motivations in physics, finance or biology (see for instance \cite{Kac,Zauderer,Renshaw,Hadeler,Chaturvedi,Rosetto,Gruber} and references within). 

Meanwhile, the search for efficient Markov Chain Monte Carlo (MCMC) methods led to the development of so-called  rejection-free or lifted chains (see e.g. \cite{Turitsyn,Diaconis2000,Krauth} and references within). In this context, the persistent walk has been a toy model to understand the efficiency of these algorithms, especially when compared to the reversible simple walk \cite{Diaconis2000,DiaconisMiclo,Monmarche2013}. For instance, correctly scaled, the persistent walk shows a ballistic behaviour, which means its expected distance to its initial position after $K$ steps is of order $K$, while the simple walk shows a diffusive behaviour, moving to a distance $\sqrt{K}$ after $K$ steps. Since the efficiency of the MCMC schemes is related to the speed at which the space is explored, this is an argument in favour of non-reversible kinetic processes. The model being simple, it is even possible to determine the optimal $\alpha$ (in the sense that it gives the maximal rate of convergence toward equilibrium on the periodic torus $\Z/(N\Z)$), which turns to be of order $1/N$ \cite{Monmarche2013}. This is consistent, as $N$ goes to infinity, with the ballistic scaling that yields the telegraph process (by contrast, if $\alpha$ is constant with $N$ and if time is accelerated by $N^2$, the persistent walk converges to the Brownian motion).

Of course, both the persistent walk and the telegraph process sample the uniform measure in dimension one, which is not of practical interest. These last years, the telegraph process has been extended to several continuous-space processes, such as the Zig-Zag sampler \cite{BierkensFearnheadRoberts,BierkensDuncan,BierkensRobertsZitt,BierkensRoberts} or the Bouncy Particle Sampler \cite{PetersdeWith,MonmarcheRTP,DurmusGuillinMonmarche2018,Doucet}, which are velocity jump processes designed to target any given distribution in any dimension. Many variants like randomized bounces \cite{RobertWu,Michelforward} are currently being developped and we refer to the review \cite{DoucetPDMCMC} for more details, considerations and references on this vivid topic.

The present paper is concerned with similar extensions, but conducted at the level of the persistent walk rather than of the continuous  kinetic process. Or, from another viewpoint, we are interested in persistent walks, but through the prism of MCMC sampling rather than kinetic theory. The  motivations are the following: first, the discrete chain yields some insights on their continuous-time limits (for instance, we will see that the Zig-Zag process can be seen as the continuous limit of a Gibbs algorithm). Second, used in an MCMC scheme on $\Z^d$, a persistent walk shares, as will be detailed in this work, the following advantages with its continuous counterparts: thinning, factorization and ballistic behaviour. Finally, although continuous-time velocity jump processes can sometimes be sampled exactly thanks to thinning methods, it is not necessarily the case for  mixed diffusion/jump  kinetic samplers (see Section \ref{Section-Schema-Euler}), in which case the corresponding chain obtained through an integration scheme (say, Euler scheme) is a discrete-time kinetic  walk.

The rest of the paper is organized as follows. We start in Section \ref{Section-Dim1} with the definition of an analogous on $\Z$ of the Zig-Zag process on $\R$ (or, equivalently, of the persistent walk but in a general potential landscape). The simplicity of the chain allows an elementary study of its ergodicity, of its metastable behaviour at small temperature via an Eyring-Kramers formula and of its convergence toward the continuous Zig-Zag process on $\R$ under proper scaling.  Section~\ref{Section-general} is a general and informal discussion about kinetic walks on $\R^d$,  their simulation, invariant measures and continuous-time scaling limits. Finally, the last two sections present two particular applications, which are the main contributions of this work: the discrete Zig-Zag walk in a  general potential in Section \ref{subSect-ZZd} and  hybrid jump/diffusion kinetic samplers with a numerical integrator in Section~\ref{Section-Schema-Euler}. Although related by their motivation (the understanding of sampling with kinetic walks), the four sections are sufficiently independent to be read separately one from the other. The definition of the Zig-Zag walk associated to a general potential  and all the results on this topic in Section~\ref{Section-Dim1} and Section \ref{subSect-ZZd}  are new. The specific numerical scheme introduced in Section~\ref{Section-Schema-Euler} is also new, although straightforwardly obtained from the well-known general method of Strang splittings.

\bigskip

\noindent \textbf{Notations.} If $x,y\in\R^d$, we denote $x\cdot y$ their scalar product and $|x|=\sqrt{x\cdot x}$. The Dirac mass at $x$ is denoted $\delta_x$ and $\1_A$ is 1 if $A$ and 0 else. For $r\in\R$, $(r)_+ = \max(r,0)$. The set of $\mathcal C^k$ functions on $\R^d$ with compactly supported supported is denoted $\mathcal C^k_c(\R^d)$. The Gaussian distribution on $\R^d$ with mean $m$ and variance $\Sigma^2$ is denoted $\mathcal N(m,\Sigma^2)$. We denote respectively $\mathcal P(E)$, $\mathcal M(E)$ and $\mathcal M_b(E)$   the sets of probability measures,  measurable functions and bounded measurable functions on a measurable space $E$, and for $\mu\in\mathcal P(E)$ and $f\in L^1(\mu)$ we write $\mu f = \mu (f) = \int f\dd \mu$.  When $(X_t^\varepsilon)_{t\geqslant 0}$ for $\varepsilon>0$ and $(Y_t)_{t\geqslant 0}$ are  c\'adl\'ag processes on  $\R^d$, we write $(X_t^\varepsilon)_{t\geqslant 0}\overset{law} {\underset{\varepsilon \rightarrow 0}\longrightarrow} (Y_t)_{t\geqslant 0}$ for the convergence in law in the Skorohod topology. Recall a sequence $(x_n)_{n\in\N}$ of càdlàg functions from $\R_+$ to $\R^d$ is said to converge to $x$ if, on all finite time interval, it converges uniformly up to a uniformly small change of time, i.e. if there exists a sequence $(\gamma_n)_{n\in\N}$ with $\gamma_n : \R_+ \rightarrow \R_+$ increasing so that $\sup_{t\in[0,T]}(|x_n(\gamma_n(t))-x(t)|+|\gamma_n(t)-t|)\rightarrow 0$ as $n\rightarrow \infty$ for all $T>0$.


\section{The Zig-Zag walk on $\Z$}\label{Section-Dim1}

Let $U:\Z \rightarrow \R$ be such that $\mathcal Z = \sum_{x\in\Z} \exp(-U(x)) < +\infty$,  $\pi(x) = \exp(-U(x))/\mathcal Z$ be the associated Gibbs distribution and $\mu(x,v) = \pi(x)/2$ for $v\in\{-1,1\}$ and $x\in\Z$. We consider the Markov chain $(X_k,V_k)_{k\in\N}$ on $\Z\times\{-1,1\}$ with transitions
\[(X_{k+1},V_{k+1}) \ = \ \left\{\begin{array}{ll}
(X_k + V_k,V_k) & \text{with probability }\min\po \frac{\pi(X_k+V_k)}{\pi(X_k)},1\pf\\
(X_k,-V_k) & \text{else,}
\end{array}\right.\]
which we call the  Zig-Zag walk on $\Z$. This transition can be seen as the composition of two Markov transitions. Indeed, consider on $\Z\times\{-1,1\}$ the Markov kernel given by $x,v\mapsto \delta_{(x+v,-v)}$. Since $(y,w)=(x+v,-v)$ implies that $(x,v)=(y+w,-w)$, this kernel is symmetric. If a Metropolis accept/reject step with target measure $\mu$ is added, the transition of the resulting chain is simply
\[(Y_{k+1},W_{k+1}) \ = \ \left\{\begin{array}{ll}
(Y_k + W_k,-W_k) & \text{with probability }\min\po \frac{\pi(Y_k+W_k)}{\pi(Y_k)},1\pf\\
(Y_k,W_k) & \text{else.}
\end{array}\right.\]
By construction of the Metropolis-Hastings algorithm, this transition leaves $\mu$ invariant. Now if we compose this transition with the deterministic transition $(Y_{k+1},W_{k+1}) = (Y_k,-W_k)$, which obviously leaves $\mu$ invariant, we obtain the initial chain. We have thus obtained that $\mu$ is invariant for the Zig-Zag walk. Note however that, although both intermediate transition kernels are reversible with respect to $\mu$, their composition is not. Indeed, $\mathbb P((X_{2},V_2)=(X_0,V_0))=0$ for all initial condition.

The chain is clearly irreducible
, and it is periodic. Indeed,  $(-1)^{X_k}V_k= (-1)^{k+X_0}V_0$ for all $k\in\N$, so that, if $X_k$ is even with $V_k=1$ or $X_k$ is odd with $V_k=-1$, then $X_{k+1}$ is odd with $V_{k+1}=1$ or $X_{k}$ is even with $V_{k+1}=-1$. In particular, the period is even. If $U$ admits a strict local minimum $x_0$ then there is a path of length 2 with strictly positive probability from $(x_0,1)$ to itself (which is $(x_0,1)\rightarrow (x_0,-1)\rightarrow (x_0,1)$), so that the period is exactly 2, but this may not be the case in general. For instance, with $U(k) = |\lfloor k/2\rfloor|$, the reader can check that the period is 4. 

In the following, we prove an ergodic Law of Large Number and a Central Limit Theorem (CLT) in Theorem \ref{Thm-ZZ1-CLT} (in the unimodal case; other cases are treated in any dimension in Section \ref{SubSectionTCLZZZd}), an Eyring-Kramers formula in Theorem~\ref{Thm-ZZ1-EK} and the convergence toward the continuous Zig-Zag process in Theorem~\ref{Thm-Dim1-scaling}.


\subsection{Asymptotic results}

Although the chain is already quite simple, let us focus for now on the case where $\pi$ is unimodal. In that case, and similarly to the continuous-time case \cite{BierkensDuncan}, ergodicity  can be established through elementary considerations on renewal chains.

\begin{thm}\label{Thm-ZZ1-CLT}
Suppose that $U$ is decreasing on $\ccf -\infty,0\ccf$ and increasing on $\cco 0,+\infty\cco$, and let $f\in L^1(\mu)$. Then, for all initial conditions, almost surely,
\[\frac1n\sum_{k=0}^{n-1} f(X_k,V_k) \ \underset{n\rightarrow +\infty}\longrightarrow \ \mu(f)\,.\]
Moreover, denoting $g(x)=f(x,1)+f(x,-1)$ and
\[F(x) = \frac12 g(x) + \1_{x\geqslant 1}\sum_{i=1}^{x-1} g(i) + \1_{x\leqslant -1}\sum_{i=x+1}^{-1} g(i)\,,\]
suppose that $M_f:=\mathbb E_\pi(g(X)F(X))<\infty$ and that $\mu(f)=0$. Then
\[\frac1{\sqrt{n}}  \sum_{k=0}^{n-1} f(X_k,V_k)  \  \overset{law} {\underset{n\rightarrow +\infty}\longrightarrow}\ \mathcal N(0,\sigma_f^2)\,, \]
with some explicit variance $\sigma_f^2 \leqslant 3M_f$.
\end{thm}

\begin{proof}
Consider first the case where $(X_0,V_0)=(0,1)$ and denote $T_1 = \inf\{n\in\N\ :\ X_{n+1}=X_n\}$, $T_2 = \inf\{n\in\N\ :\ X_{2T_1+n+2}=X_{2T_1+1+n}\}$. The monotonicities of $U$ implies that almost surely $X_n$ increases for $n\in\cco 0,T_1\ccf$, decreases for $n\in\cco T_1+1,2T_1+1+T_2\ccf$ with $X_{2T_1+1} = 0$, and finally, denoting $S_1 = 2(T_1+T_2+1)$, increases for $n\in\cco 2T_1+2+T_2,S_1\ccf$ with $(X_{S_1},V_{S_1}) = (0,1)$ (cf. Fig. \ref{figTCL}). Remark that
\[\mathbb P\po T_1 \geqslant k\pf \ = \ \prod_{j=0}^{k-1} e^{U(j) - U(j+1)} \ = \ e^{U(0) - U(k)} \,,\]
so that $T_1 < \infty$ almost surely (since we assumed that $\mathcal Z <\infty$, $U$ necessarily goes to $\infty$ at $\infty$). The same goes for $T_2$, hence for $S_1$.  By the strong Markov property, $(X_n,V_n)_{n\geqslant S_1}$ has the same law as $(X_n,V_n)_{n\in\N}$ and is independent from $(X_n,V_n)_{n\in\cco 0,S_1-1\ccf}$.   Denote $S_0=0$ and, for all $n\in\N$, $S_{n+1} = \inf\{k>S_n\ : \ (X_k,V_k)=(0,1)\}$ and, given a function $f\in L^1(\mu)$,
\[A_n \ = \ \sum_{k=S_n}^{S_{n+1}-1} f(X_k,V_k)\,.\]
The $A_n$'s are i.i.d. and
\begin{eqnarray*}
\mathbb E \po \left|\sum_{k=0}^{2T_1+1} f(X_k,V_k)\right |\pf & \leqslant & \sum_{j\in\N} \sum_{k\leqslant j} \mathbb P\po T_1 = j\pf \po |f(k,1)|+|f(k,-1)|\pf \\
& = & \sum_{k\in\N} e^{U(0) - U(k)} \po |f(k,1)|+|f(k,-1)|\pf \ < \ +\infty\,.
\end{eqnarray*}
The sum for $k\in\cco 2 T_1+2,S_1-1\ccf$ is treated the same way, so that $\mathbb E|A_0| < \infty$ and 
\[\mathbb E\po A_0\pf \ =\  \sum_{k\in\Z} e^{U(0) - U(k)} \po f(k,1)+f(k,-1)\pf \ = \ \lambda \mu(f)\]
with $\lambda = 2 e^{U(0)}\mathcal Z$. The proof then follows from classical renewal arguments, which we recall for completeness. Considering the case $f=1$, the law of large numbers implies that $S_n/n$ converges almost surely toward $\lambda$ as $n$ goes to infinity. For $n\in\N$ set $K(n) = \sup\{k\in\N\ : \ S_k \leqslant n\}$.  If $f$ is positive then for all $n\in\N$,
\[\frac{1}{n}\sum_{j=0}^{K(n)} A_j \ \leqslant \ \frac{1}{n}\sum_{k=0}^{n-1} f(X_k,V_k)  \ \leqslant \ \frac{1}{n}\sum_{j=0}^{K(n)+1} A_j \,.\]
Applied with $f=1$, this reads
\[\frac{K(n)}{n}\times \frac{S_{K(n)}}{K(n)} \ \leqslant \ 1 \ \leqslant \ \frac{K(n)+1}{n}\times \frac{S_{K(n)+1}}{K(n)+1}\,. \]
Since $K(n)$ almost surely goes to infinity with $n$, we get that $K(n)/n$ almost surely converges to $1/\lambda$. Applied again with a general positive $f$, now,
\[\frac{K(n)}{n}\times \frac{1}{K(n)}\sum_{j=0}^{K(n)} A_j \ \leqslant \ \frac{1}{n}\sum_{k=0}^{n-1} f(X_k,V_k)  \ \leqslant \ \frac{K(n)+1}{n}\times \frac{1}{K(n)+1}\sum_{j=0}^{K(n)+1} A_j \,,\]
and letting $n$ go to infinity concludes. If $f$ is not positive, the same conclusion follows from the decomposition $f=(f)_+ - (-f)_+$. 

\begin{figure}
\begin{center}
\includegraphics[scale=0.4]{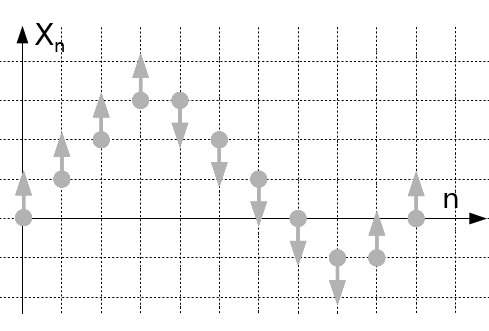}
\caption{The trajectory between times $S_{0}$ and $S_{1}$.}\label{figTCL}
\end{center}
\end{figure}

Now, consider the case of any general initial condition $(X_0,V_0)=(x,v)$, and let $R = \inf\{n\in\N\ :\ (X_n,V_n) = (0,1)\}$. By similar arguments as above, $R<\infty$ almost surely so that $n^{-1}\sum_{k< R} f(X_k,V_k)$ almost surely goes to zero, while by the Markov property, $n^{-1} \sum_{R\leqslant k \leqslant n} f(X_k,V_k)$ converges toward $\mu(f)$, which concludes.

The proof of the CLT is similar, and we refer to \cite[Lemma 4]{BierkensDuncan} to get that, if $\mu(f)=0$,
\[\sqrt{\frac{\lambda}{n}}\sum_{j=0}^{K(n)}  A_j  \ \overset{law}{\underset{n\rightarrow\infty}\longrightarrow} \ \mathcal N(0,\sigma_A^2)\,,\]
provided that $\sigma_A^2 := \mathbb E\po A_0^2\pf  <\infty$. Now, even if $(X_0,V_0)\neq(0,1)$, as before,
\[\frac1{\sqrt n} \po \sum_{k=0}^{n-1}  f(X_k,V_k)  - \sum_{k=0}^{K(n)}  A_k\pf \ \overset{\mathbb P}{\underset{n\rightarrow\infty}\longrightarrow}  \ 0\,,\]
hence, provided that $\sigma_A^2 <\infty$,
\[\frac1{\sqrt n} \sum_{k=0}^{n-1}  f(X_k,V_k) \ \overset{law}{\underset{n\rightarrow\infty}\longrightarrow} \ \mathcal N(0,\sigma_A^2/\lambda)\,.\]
  Decompose $A_0=f(0,1)+A_0'+f(0,-1)+A_0''$ with
\[A_0' \ = \ \sum_{k=1}^{2T_1} f(X_k,V_k)\,,\qquad A_0'' \ = \ \sum_{k=2T_1+2}^{S_1-1} f(X_k,V_k)\,,\]
and remark that by the Markov property, $A_0''$ is independent from $A_0$. Compute
\begin{eqnarray*}
\mathbb E\po (A_0')^2\pf & = & \sum_{k\in\N_*} \mathbb P (T_1=k) \sum_{i=1}^k\sum_{j=1}^k g(i)g(j)\\
& = & \sum_{i\in\N_*}\sum_{j\in\N_*} g(i)g(j)  \mathbb P (T_1\geqslant i\vee j) \\
& = &   \sum_{i\in\N_*} g(i) e^{U(0)-U(i)} \po g(i) +2\sum_{j=1}^{i-1}g(j)\pf\\
& = & \lambda\sum_{i\in\N_*} g(i) F(i) \pi(i)  \,.
\end{eqnarray*}
 The case of $A_0''$ is similar and, using that $1/\lambda = \pi(0)/2$, we get
 \begin{eqnarray*}
\frac1{\lambda}\mathbb E\po A_0^2\pf &\leqslant & \frac3{\lambda}\mathbb E\po (A_0')^2 + (A_0'')^2 + g^2(0)\pf \ = \ 3 \sum_{x\in\Z} g(x) F(x) \pi(x)\,, 
\end{eqnarray*}
which concludes. In fact $\sigma_f^2=\sigma_A^2/\lambda$ can be computed since $\mathbb E(A_0') = \lambda \sum_{x\in\N_*} g(x)\pi(x)$, and similarly for $A_0''$.
\end{proof}

For $N\in\N_*$ and $t_1<\dots<t_N$, considering $K(t_i) = \sup\{k\in\N\ : \ S_k \leqslant \lfloor n t\rfloor  \}$ and decomposing $\sum_{n=1}^{K(t_N)-1} A_n = \sum_{i=1}^N \sum_{n=K(t_i)}^{K(t_{i+1})-1} A_n$, the previous elementary proof is easily extended to obtain a functional CLT, namely the convergence
\[\po \frac1{\sqrt{n}}\sum_{k=0}^{\lfloor n t\rfloor}\po  f(X_k,V_k) - \mu(f)\pf \pf_{t\geqslant 0}\  \overset{law} {\underset{n\rightarrow +\infty}\longrightarrow}\ (\sigma_f B_t)_{t\geqslant0}\]
where $(B_t)_{t\geqslant0}$ is a one-dimensional Brownian motion. See also Section \ref{SubSectionTCLZZZd}.

\subsection{Metastability}

We now consider the question of escape times from local minima at low temperature, as in \cite{MonmarcheRTP} for the Zig-Zag process on $\R$. Recall that a random variable $G$ on $\R$ is said to be stochastically larger than a random variable $F$ on $\R$  if $\mathbb P\po G<t\pf \leqslant \mathbb P\po F<t\pf$ for all $t\in\R$. In that case we write $F\overset{sto}{\leqslant} G$. 

\begin{thm}\label{Thm-ZZ1-EK}
Let $U:\mathbb Z\rightarrow \R$ and, for all $\varepsilon>0$, let $(X^\varepsilon_k,Y^\varepsilon_k)_{k\in\N}$ be the persistent walk on $\Z$ associated to $U/\varepsilon$ and with  initial condition $(0,1)$. Suppose that $U(0)=0$ and that $U$ is decreasing on $\ccf -\infty,0\ccf$ and increasing on $\cco 0,+\infty\cco$. Let $a<\alpha\leqslant 0 \leqslant \beta < b$ be such that  $\cco \alpha,\beta\ccf = \{k\in\Z,\ U(k)= 0\}$, and let
\[\tau_\varepsilon \ = \ \inf\{n\in\N,\ X_n^\varepsilon \notin \ccf a,b\cco\}\,.\] 
Then
\begin{eqnarray}\label{Eq-EyringKramer-thm}
\mathbb E \po \tau_\varepsilon \pf &=& e^{E_1/\varepsilon  }\po \frac{2(\beta-\alpha+1)}{1 + \1_{U(a)=U(b)} }  + \underset{\varepsilon\rightarrow 0}{\mathcal O}\po e^{-E_2/\varepsilon }\pf +  \1_{U(a)\neq U(b)} \underset{\varepsilon\rightarrow 0}{\mathcal O}\po e^{-E_3/\varepsilon }\pf \pf\,,
\end{eqnarray}
with $E_1  = \min\po U(a), U(b)\pf$, $E_2 = \min\po U(\alpha-1), U(\beta+1)\pf $ and $E_3 = |U(a)-U(b)|$. Moreover, $\tau_\varepsilon / \mathbb E(\tau_\varepsilon)$ converges in law as $\varepsilon$ vanishes to an exponential random variable with parameter 1, and
\[\mathbb P(X_{\tau_\varepsilon}^\varepsilon=a) \  \underset{\varepsilon\rightarrow 0}\longrightarrow \  \frac12 \po 1 + \1_{U(a)\leqslant U(b)} - \1_{U(b)\leqslant U(a)}\pf\,.\]
Finally, for all $\varepsilon>0$, 
\[2(\beta-\alpha+1)\po  G_\varepsilon - 1\pf \ \overset{sto}{\leqslant} \ \tau_\varepsilon \ \ \overset{sto}{\leqslant} \  2(b-a+1) G_\varepsilon\]
where $G_\varepsilon$ is a geometric random variable with  parameter
\[ e^{ - U(b)/\varepsilon} +  e^{   - U(a)/\varepsilon} - e^{  -\po  U(b) + U(a)\pf/\varepsilon} \,.\]
\end{thm}

\begin{proof}
The proof is similar to the one appearing in \cite{MonmarcheRTP}. To alleviate notations, we only write $X_k,V_k$ and $\tau$ for $X_k^\varepsilon,V_k^\varepsilon$ and $\tau_\varepsilon$. Like in the previous proof, set $S_0=0$ and, by induction, $S_{n+1} = \inf\{k>S_n,\ (X_k,V_k)=(0,1)\}$. For all $n\in\N$, let $\tilde S_n = \inf\{k>S_n,\ (X_k,V_k)=(0,-1)\}$, and let $K=\inf\{n\in\N,\ S_n>\tau\}$. Keep Figure \ref{figTCL} in mind. By the strong Markov property, $K$ follows a geometric distribution with parameter
\begin{eqnarray}
p  \ := \  \mathbb P \po \tau  < S_1\pf \notag
& = & \mathbb P \po \tau  < \tilde S_0\pf + \mathbb P \po \tau  >\tilde S_0\pf \mathbb P \po \tau  <   S_1 | \tau  > \tilde S_0\pf \notag\\
& = & e^{ - U(b)/\varepsilon} + \po 1 - e^{   - U(b)/\varepsilon}\pf e^{   - U(a)/\varepsilon} \notag\\
& = & e^{-E_1/\varepsilon} \po 1 + \1_{U(a)=U(b)} + \underset{\varepsilon\rightarrow 0}{\mathcal O}\po e^{-E_1/\varepsilon }\pf  + \1_{U(a)\neq U(b)}  \underset{\varepsilon\rightarrow 0}{\mathcal O}\po e^{-E_3/\varepsilon }\pf  \pf \notag \,,
\end{eqnarray}
and 
\[\mathbb P \po X_\tau = b \pf \ = \  \mathbb P \po X_\tau = b \ | \ \tau < S_1 \pf \ = \  \frac1p  e^{ - U(b)/\varepsilon}\,,\]
which indeed converges as $\varepsilon$ vanishes to 0 if $U(a)<U(b)$, 1 if $U(a)>U(b)$ and $1/2$ if $U(a)=U(b)$. Decomposing $\tau\ = \ \tau - S_{K-1} + \sum_{i=1}^{K-1} (S_i - S_{i-1})$, remark that almost surely $2(\beta-\alpha+1) \leqslant S_{i}-S_{i-1}\leqslant 2(b-a+1)$ for all $i<K$ and $\tau - S_{K-1} < 2(b-a+1)$, which proves the last claim of the theorem. Besides, again by the strong Markov property, conditionally to $K$, $(S_i-S_{i-1})_{i\in\cco 1,K-1\ccf}$ are i.i.d. random variables independent from $K$, so that
\begin{eqnarray}\label{Eq-EyringKramer1}
\mathbb E(\tau)& = & \mathbb E\po \tau - S_{K-1}\pf + \mathbb E(K) \mathbb E \po S_1 \ |\ \tau>S_1\pf\,.
\end{eqnarray} 
Since  $0 \leqslant \tau-S_{K-1} \leqslant 2(b-a)$ almost surely, 
\begin{eqnarray}\label{Eq-EyringKramer2}
\left|\mathbb E(\tau-S_{K-1})\right| \ \leqslant \ 2(b-a) \ = \ e^{E_1/\varepsilon  } \underset{\varepsilon\rightarrow 0}{\mathcal O}\po e^{-E_2/\varepsilon }\pf \,,
\end{eqnarray}
where we used that $E_1\geqslant E_2$. For $\varepsilon$ small, the most likely trajectory of the process between times $0$ and $S_1$ is the following: starting from $(0,1)$, it deterministically goes to $(\beta,1)$, then jumps to $(\beta,-1)$ with high probability, then deterministically goes to $(\alpha,-1)$ and jumps to $(\alpha,1)$ with high probability before going back to $(0,1)$ deterministically. More precisely, $S_1\geqslant 2(\beta - \alpha +1)$ almost surely, and
\begin{eqnarray*}
\mathbb P\po S_1= 2(\beta - \alpha +1)\pf 
& = & \po 1 - e^{ \po U(\alpha)-U(\alpha-1)\pf/\varepsilon}\pf \po 1 - e^{ \po U(\beta)-U(\beta+1)\pf/\varepsilon}\pf  \\
& = & 1 + \underset{\varepsilon\rightarrow 0}{\mathcal O}\po e^{-E_2/\varepsilon} \pf \,.
\end{eqnarray*}
On the other hand, conditionally to $\tau>S_1$, almost surely, $S_1 \leqslant 2(b-a)$, so that
\[\mathbb E \po S_1 \1_{\tau>S_1>2(\beta-\alpha +1)}\pf \ \leqslant \ 2(b-a) \mathbb P\po S_1\neq  2(\beta - \alpha +1)\pf \ =\  \underset{\varepsilon\rightarrow 0}{\mathcal O}\po e^{-E_2/\varepsilon} \pf \,.\]
Thus, we get that
\[\mathbb E \po S_1 \ |\ \tau>S_1\pf \ = \ \frac1{1-p}\mathbb E \po S_1 \1_{\tau>S_1}\pf\ = \  2(\beta-\alpha+1)  + \underset{\varepsilon\rightarrow 0}{\mathcal O}\po e^{-E_2/\varepsilon }\pf\,,\]
where we used again that $E_1\geqslant E_2$. Using in \eqref{Eq-EyringKramer1} this estimate together with \eqref{Eq-EyringKramer2} and the fact that $\mathbb E(K) = 1/p$ concludes the proof of the Eyring-Kramers formula \eqref{Eq-EyringKramer-thm}.

Finally, $K$ being an exponential random variable whose parameter vanishes with $\varepsilon$, $pK$ converges in law toward an exponential random variable with parameter 1. By the Markov inequality, for any $\delta>0$,
\begin{eqnarray*}
\mathbb P \po \left|\frac1K \sum_{i=1}^{K-1} (S_i - S_{i-1}) - \mathbb E\po S_1\ |\ \tau>S_1\pf\right|>\delta\pf & \leqslant &\frac1{\delta^2} \mathbb E\po S_1^2\ |\ \tau>S_1\pf\ \mathbb E\po \frac1K\pf\\
& \leqslant & \frac{4(b-a)^2}{\delta^2} \mathbb E\po \frac1K\pf \ \underset{\varepsilon\rightarrow0}\longrightarrow \ 0\,.
\end{eqnarray*} 
Since $\mathbb E\po S_1\ |\ \tau>S_1\pf$  and $p\mathbb E(\tau) $ both converges toward $2(\beta-\alpha+1)$ as $\varepsilon$ vanishes,
\[\frac1{Kp\mathbb E(\tau)} \sum_{i=1}^{K-1} (S_i - S_{i-1}) \ \underset{\varepsilon\rightarrow 0}{\overset{\mathbb P}{\longrightarrow}} \ 1\,. \]
From Slutsky's theorem, $\sum_{i=1}^{K-1} (S_i - S_{i-1}) /\mathbb E(\tau)$ converges in law to an exponential random variable with parameter 1 as $\varepsilon\rightarrow 0$. Finally, $|(\tau- S_{K-1})/\mathbb E(\tau)| \leqslant 2(b-a)/\mathbb E(\tau)$ almost surely goes to zero, which concludes.

\end{proof}

The simulated annealing chain obtained by considering a non-constant temperature $(\varepsilon_k)_{k\in\N}$ and a potential $U$ with possibly several local minima could also be studied by similar arguments as in \cite[Theorem 3.1]{MonmarcheRTP} to get a necessary and sufficient condition on the cooling schedule for convergence in probability toward the global minima of $U$.

\subsection{Continuous scaling limit}

The (continuous-time) Zig-Zag process on $\R$ associated to a potential $H\in\mathcal C^1(\R)$ (also known as the (integrated) telegraph or run-and-tumble process) is the  Markov process on $\R\times\{-1,1\}$ with generator
\[L f(y,w) \ = \ w\partial_y f(y,w) + \po w H'(x)\pf_+ \po f(y,-w) - f(y,w)\pf\,.\]
In other words, it is a piecewise deterministic Markov process that, starting from an initial condition $(y,w)$, follows the flow $(Y_t,W_t) = (y+tw,w)$ up to a random time $T$ with distribution $\mathbb P(T>t) = \exp(-\int_0^t (wH'(y+sw))_+ \dd s)$, at which point $(Y_T,W_T)=(y+Tw,-w)$, after which it follows again the deterministic flow up to a new random jump time, etc. 

\begin{thm}\label{Thm-Dim1-scaling}
For $H\in\mathcal C^2(\R)$ that goes to infinity at infinity, for all $\varepsilon>0$, define $U_\varepsilon : \Z\mapsto \R$ by $U_\varepsilon(k) = H(\varepsilon k)$ for all $k\in\Z$. Let $(X^\varepsilon_k,V^\varepsilon_k)_{k\in\N}$ be the persistent walk on $\Z$ associated to $U_\varepsilon$ and with some initial condition $(x_0^\varepsilon,v_0)$. Suppose that $\varepsilon x_0^\varepsilon$ converges to some $x_0^*\in\R$ as $\varepsilon$ vanishes. Then,
\[\po \varepsilon X^\varepsilon_{\lfloor t/\varepsilon\rfloor},V^\varepsilon_{\lfloor t/\varepsilon\rfloor}\pf_{t\geqslant 0} \ \overset{law} {\underset{\varepsilon \rightarrow 0}\longrightarrow}\ \po Y_{t},W_{t}\pf_{t\geqslant 0}\,,\]
where $(Y_t,W_t)_{t\geqslant 0}$ is a  Zig-Zag process on $\R$ associated to $H$ and with $(Y_0,W_0)=(x_0^*,v_0)$.
\end{thm}

\begin{proof}
Denote $T^\varepsilon_1 = \varepsilon\inf\{n\in\N\ : \ X_{n+1}^\varepsilon = X_n^\varepsilon\}$. Its cumulative function is 
\[F_{x_0,v_0}^\varepsilon(t) \ :=\  \mathbb P\po T^\varepsilon_1 \leqslant t \pf \ = \ 1 - \prod_{k=1}^{\lfloor t/\varepsilon\rfloor} \exp  \po-\po U_\varepsilon(x_0^\varepsilon+kv_0) - U_\varepsilon(x_0^\varepsilon+(k-1)v_0)\pf_+\pf\,.\]
From
\[\sum_{k=1}^{\lfloor t/\varepsilon\rfloor} \po U_\varepsilon(x_0^\varepsilon+kv_0) - U_\varepsilon(x_0^\varepsilon+(k-1)v_0)\pf_+ \ \underset{\varepsilon\rightarrow 0}\longrightarrow \ \int_0^t \po v_0 H'(x_0^*+sv_0)\pf_+ \dd s\,,\]
we get that $T_1^\varepsilon$ converges in law as $\varepsilon$ vanishes to a random variable $T_1^0$ with cumulative function
\[F^0_{x_0,v_0}(t) \ = \ 1 - \exp  \po -\int_0^t\po v_0 H'(x_0^*+sv_0)\pf_+ \dd s\pf \,. \]
Remark that
\[\int_0^t \po v_0 H'(x_0^*+sv_0)\pf_+ \dd s \ \geqslant \ \int_0^t v_0 H'(x_0^*+sv_0) \dd s = H(x_0^* + t v_0) - H(x_0^*) \underset{t\rightarrow \infty}\longrightarrow + \infty\,,\]
so that $T_1^0$ is almost surely finite, and similarly for $T_1^\varepsilon$ for all $\varepsilon >0$.
In particular,
\[\po \varepsilon X_{T_1^\varepsilon},V_{T_1^\varepsilon}^\varepsilon,T_1^\varepsilon\pf \ = \ \po \varepsilon x_0 + T_1^\varepsilon,-v_0,T_1^\varepsilon\pf  \ \overset{law}{\underset{\varepsilon\rightarrow 0}\longrightarrow}\ \po x_0^*+T_1^*,-v_0,T_1^*\pf \,.\]
Let $(A_j)_{j\in\N}$ be an i.i.d. sequence of random variables uniformly distributed over $[0,1]$. For all $\varepsilon\geqslant 0$, set $(Z_0^\varepsilon,R_0^\varepsilon,S_0^\varepsilon)= (\varepsilon x_0^\varepsilon,v_0,0)$ (with, in the case where $\varepsilon=0$, $Z_0^0= x_0^*$). Suppose by induction that, for some $n\in\N$, $(Z_n^\varepsilon,R_n^\varepsilon,S_n^\varepsilon)$ has been defined for all $\varepsilon \geqslant 0$ and is independent from $(A_j)_{j\geqslant n}$. Then, for all $\varepsilon\geqslant 0$,   set 
\begin{eqnarray*}
S_{n+1}^\varepsilon & = &  S_n^\varepsilon + \po F_{Z_n^\varepsilon,R_n^\varepsilon}^\varepsilon\pf^{-1}(A_n)\\ 
Z_{n+1}^\varepsilon & =&  Z_n^\varepsilon + (S_{n+1}^\varepsilon - S_{n}^\varepsilon ) R_{n}^\varepsilon
\end{eqnarray*}
and $R_{n+1}^\varepsilon  = - R_{n}^\varepsilon  $.  Remark that, for all $t\geqslant 0$, $(x,v) \mapsto F_{x,v}^\varepsilon(t)$ is continuous, uniformly in $\varepsilon$. As a consequence, for all $N\in\N$, by the previous result, almost surely, 
\begin{eqnarray}\label{Eq-CV-loi-squelette}
(Z_n^\varepsilon,R_n^\varepsilon,S_n^\varepsilon)_{n\in\cco 0,N\ccf} & \underset{\varepsilon\rightarrow 0}\longrightarrow &  (Z_n^0,R_n^0,S_n^0)_{n\in\cco 0,N\ccf} \,.
\end{eqnarray}
By construction, for all $\varepsilon>0$,
\[(Z_n^\varepsilon,R_n^\varepsilon,S_n^\varepsilon)_{n\in\cco 0,N\ccf} \ \overset{law}= \ (\varepsilon X_{T_n^\varepsilon/\varepsilon}^\varepsilon,V_{T_n^\varepsilon/\varepsilon}^\varepsilon, T_n^\varepsilon)_{n\in\cco 0,N\ccf} \]
with $T_0^\varepsilon = 0$ and, by induction, $T^\varepsilon_{k+1} = \varepsilon\inf\{k>T^\varepsilon_{k}\ : \ X_{n+1}^\varepsilon = X_n^\varepsilon\}$, and similarly
\[(Z_n^0,R_n^0,S_n^0)_{n\in\cco 0,N\ccf} \ \overset{law}= \ (Y_{J_n},W_{J_n}, J_n)_{n\in\cco 0,N\ccf} \]
with $J_0=0$ and by induction $J_{k+1} = \inf\{t>J_k,\ W_t = - W_{J_k}\}$. At this point, we have thus proved that the skeleton chain of the persistent walk (namely the persistent walk observed at its jump times, and those jump times) converges in law toward the skeleton chain of the Zig-Zag process (namely the process observed at its jump times, and those jump times). The convergence of the full chain is then a consequence from the fact that the latter is a deterministic function of its skeleton chain, as we detail now.


Note that $(S_n^0)_{n\in\N}$ has the same distribution as $(J_n)_{n\geqslant 0}$. Moreover, for any $t\geqslant 0$ and
 for all $s\in[0,t]$, $|Y_s- x_0^*|\leqslant t$, so that the jump rate of the Zig-Zag process is bounded for times $s\in[0,t]$ by $\omega(t) = \sup_{x\in[x_0^*-t,x_0^*+t]}|H'(x)|$, which is finite. In particular $\sup\{n\in\N,\ J_n<t\}$ the number of jumps of the Zig-Zag process on $[0,t]$ is stochastically smaller than a Poisson process with rate $\omega(t)$, hence is almost surely finite. 
As a consequence, almost surely $S_n^0 \rightarrow +\infty$ as $n\rightarrow \infty$, and similarly for $S_n^\varepsilon$ for all $\varepsilon\geqslant 0$.

Now the continuous-time processes are obtained by interpolating the skeleton chains. For all $\varepsilon\geqslant 0$ and $n\in\N$, set $\tilde R_t^\varepsilon = R_n^\varepsilon$ for all $t\in[S_n^\varepsilon,S_{n+1}^\varepsilon[$. For all $t\geqslant 0$, set $\tilde Z_t^0 = x_0^* + \int_0^t \tilde R_s^0 \dd s$ and remark that, by construction, $\tilde Z_{T_n}^0 =Z_n^0$ for all $n\in\N$. Finally, for all $\varepsilon>0$, all $n\in\N$ and all $k\in\cco T_{n}/\varepsilon,T_{n+1}/\varepsilon\ccf$, set 
\[\tilde Z_{k\varepsilon}^\varepsilon \ = \  Z_{n}^\varepsilon + \frac{ k\varepsilon- T_n^\varepsilon}{T_{n+1}^\varepsilon-T_n^\varepsilon} \po Z_{n+1}^\varepsilon - Z_n^\varepsilon\pf\]
and $\tilde Z_t^\varepsilon = \tilde Z_{k\varepsilon}^\varepsilon$ for all $t\in [k\varepsilon,(k+1)\varepsilon[$. This construction ensures that, for all $\varepsilon>0$,
\[\po \tilde Z_t^\varepsilon,\tilde R_t^\varepsilon \pf_{t\geqslant 0} \overset{law}= \po \varepsilon X^\varepsilon_{\lfloor t/\varepsilon\rfloor},V^\varepsilon_{\lfloor t/\varepsilon\rfloor}\pf_{t\geqslant 0} 
\qquad \text{and}\qquad 
\po \tilde Z_t^0,\tilde R_t^0\pf_{t\geqslant 0} \ \overset{law} {=}\ \po Y_{t},W_{t}\pf_{t\geqslant 0}\,.\]

For all $\varepsilon>0$, consider the increasing continuous change of time $\gamma_\varepsilon : \R_+\rightarrow \R_+$ given by: for all $n\in\N$, $\gamma_\varepsilon(T_n^\varepsilon) = T_n^0$ and $\gamma_\varepsilon$ is linear on $[T_n^\varepsilon,T_{n+1}^\varepsilon]$. In particular, $\tilde R_{\gamma_\varepsilon(t)}^\varepsilon = R_t^0$ for all $t\geqslant 0$ (they start at the same value and both change sign at all times $(T_n^0)_{n\geqslant 1}$).  Moreover, from the convergence of the skeleton chains, $\sup_{t\in [0,T_n^0]} \po |\gamma_\varepsilon(t) - t| + |\tilde Z_{\gamma_\varepsilon(t)}^\varepsilon -Z_t^0| \pf$ almost surely goes to $0$ as $\varepsilon$ vanishes. 
Together with the fact $T_n^0$ almost surely goes to $+\infty$, this means that for all fixed $T>0$,  $\sup_{t\in [0,T]} \po |\gamma_\varepsilon(t) - t| + |\tilde Z_{\gamma_\varepsilon(t)}^\varepsilon -Z_t^0| \pf$ almost surely goes to $0$, which concludes.

\end{proof}


\section{Kinetic walks on $\R^d$}\label{Section-general}

In the following, we will be interested in a class
of Markov chains $(X_n,V_n)_{n\in\N}$ on $\R^d\times\R^d$ for $d\in\N_*$, with transitions given by 
\begin{eqnarray}\label{Eq-transition-kinetic}
V_1 \sim p\po X_0,V_0;\cdot\pf \,,\qquad X_1 =  X_0 + \frac{\delta}{2}(V_0+V_1)
\end{eqnarray}
for  some $\delta>0$ and some kernel $p:(x,v)\in \R^d\times \R^d\mapsto p(x,v;\cdot)\in  \mathcal P(\R^d)$. We call such a chain the kinetic walk on $\R^d$ associated to $p$ with timestep $\delta$.  Up to a rescaling of the velocities, we can always consider that $\delta=1$.

This definition is close to -- but distinct from --  the definition of second-order Markov chains on $\R^d$ (sometimes also called correlated random walks like in \cite{Gruber}). Indeed, $(X_n,X_{n-1})_{n\in\N}$  is a Markov chain  if and only if $(X_n,X_n-X_{n-1})_{n\in\N}$  is, with a simple way to express the transition of one of these chains from the transition of the other. Denoting $V_n = (X_n - X_{n-1})/\delta$ would yield $X_1 = X_0 + \delta V_0$. On the contrary, consider  the chain defined in Section~\ref{Section-Dim1}, which satisfies \eqref{Eq-transition-kinetic} with $\delta = 1$. For this chain, $(X_n,X_{n-1})_{n\in\N}$ is not Markovian: if $X_n=X_{n-1}$ is at a strict local minimum of the potential $U$, then that only means that the velocity $V_n$ has changed between times $n-1$ and $n$, but it could be from $1$ to $-1$ or the converse (which we could know by looking farther in the past trajectory, for instance with the fact that $(-1)^{X_k}V_k= (-1)^{k+X_0}V_0$ for all $k\in\N$), and this affects the law of $X_{n+1}$. Our present definition is only motivated by the fact it gives a simple and unified framework for the cases studied in Sections \ref{Section-Dim1}, \ref{subSect-ZZd}  and \ref{Section-Schema-Euler}. Second-order Markov or related chains (like the discrete-time bounce sampler of \cite{RobertWu}) may be studied with the same arguments (especially concerning their continuous-time scaling limits). 
 We use the term \emph{kinetic} rather than \emph{persistent} in order to keep the latter for cases where the velocity is typically constant for some times, and this is not always the case for the different kinetic walks we will be interested in.


Note that discrete-space walks  can be seen as particular cases of walks on $\R^d$ as follows. Let $(X_n,V_n)_{n\in\N}$ be a kinetic walk on $\Z^d\times \Z^d$ with transitions given by \eqref{Eq-transition-kinetic} with  $\delta>0$ and $p:\Z^d\times\Z^d \rightarrow \mathcal P(\Z^d)$ and let $\eta,\kappa>0$. Consider the chain $(\tilde X_n,\tilde V_n)_{n\in\N}$ on $\R^d\times \R^d$ with transitions given by 
\[\mathbb P_{x_0,v_0} \po (\tilde X_1,\tilde V_1) = (x_1,v_1)\pf \ = \ \delta_{\eta\left \lfloor \frac{x_0}{\eta} + \frac{\delta}{2\kappa}(v_0+v_1) \right \rfloor }(x_1)p\po  \left \lfloor \frac{x_0}{\eta}\right \rfloor, \left  \lfloor \frac{v_0}{\kappa}\right \rfloor;\frac{v_1}{\kappa} \pf   \]
if $v_1\in \kappa\Z^d$ and zero else. In particular, whatever the initial condition, $\tilde X_n \in \eta\Z^d$ and  $\tilde V_n \in \kappa\Z^d$ for all $n\geqslant 1$. If $(\tilde X_0,\tilde V_0) = (\eta X_0,\kappa V_0)$, then $(\eta X_n,\kappa V_n)_{n\in\N}$ and $(\tilde X_n,\tilde V_n)_{n\in\N}$  have the same law. For this reason, in the rest of this section, only kinetic walks on $\R^d$ will be considered. See Section \ref{subSect-ZZd} for an example of kinetic walk on $\Z^d$.

This section is more concerned with a general and informal discussion than with rigorous results, the latter possibly requiring technical details that can be checked on explicit examples (see Section \ref{subSect-ZZd} in particular). In particular the results of this section (Proposition \ref{Prop:invariant} and Theorem \ref{Thm-Kallenberg}) are not new results.


\subsection{First examples}\label{SubSectionExKinetic}

\noindent \textbf{Example 1.} Let $U\in\mathcal C^1(\R^d)$. Then  the Hamiltonian dynamics $\partial_t (x_t,v_t) = (v_t,-\na U(x_t))$ can be discretized as
\[V_{(n+1)\delta} \ = \ V_{n\delta} - \delta \na U\po X_{n\delta} + \frac\delta 2V_{n\delta}\pf\,, \qquad X_{(n+1)\delta} \ =\  X_{n\delta} + \frac\delta 2\po V_{n\delta} + V_{(n+1)\delta} \pf\]
for some time-step $\delta>0$. This is a slight modification of the classical velocity Verlet integrator. It is a second-order scheme and, contrary to the basic Euler scheme, it is symplectic, like the Hamiltonian dynamics. From KAM theory and backward error analysis, it can be shown to conserve up to a high precision an approximate Hamiltonian, which ensures long-time stability, see \cite{HairerLubichWanner} and in particular \cite[Theorem 5.1]{HairerLubichWanner}  for long-time energy conservation.

\medskip

\noindent \textbf{Example 2.}  The Langevin diffusion (or sometimes underdamped Langevin diffusion)
\[\dd X_t \ = \ V_t\dd t \,,\qquad \dd V_t \ = \ -(\na U(X_t) - \gamma V_t)\dd t + \sqrt{2\gamma}\dd B_t\,,\]
where $\gamma>0$ and $(B_t)_{t\geqslant 0}$ is a standard Brownian motion on $\R^d$, can be approximated by similar second-order schemes (see 
\cite{Leimkuhler,BouRabee}, references within and Section \ref{Section-Schema-Euler} for more details on this topic). For instance, the Ricci-Ciccotti scheme \cite{RicciCiccotti}  reads
\begin{eqnarray*}
V_{(n+1)\delta} &=& e^{-\gamma \delta } V_{n\delta} - \po 1 - e^{-\gamma \delta } \pf \na U\po X_{n\delta} + \frac\delta 2V_{n\delta}\pf + \sqrt{\po 1 - e^{-\gamma \delta } \pf  } G_n\\
X_{(n+1)\delta} &=& X_{n\delta} + \frac\delta 2\po V_{n\delta} + V_{(n+1)\delta} \pf\,,
\end{eqnarray*}
where $(G_n)_{n\in\N}$ is an i.i.d. sequence with law $\mathcal N(0,I_d)$.

\subsection{Sampling by thinning}\label{Section-thinning}

The continous-time thinning and superposition method for sampling inhomogeneous Poisson processes, hence piecewise-deterministic Markov processes, is detailed e.g. in \cite{Thieullen2016}. See Section \ref{SubSectionStrang} for an example of application. This section is concerned with its discrete analogous, which is essentially a rejection method applied to Bernoulli random variables (see also \cite{Brill,Michelclock} on similar topics).

Suppose that the transition $p$ can be decomposed as
\[p(x,v;\cdot ) \ = \ q(x,v) p_1(x,v;\cdot ) + (1-q(x,v)) p_2(x,v;\cdot )\]
where, from a numerical point of view, computing $q$ and sampling according to $p_1$ is expensive, and sampling according to $p_2$ is not (for instance, $p_2(x,v;\cdot) = \delta_v(\cdot ) $ for persistent chains). Suppose moreover that $q(x,v) \leqslant \tilde q(x,v)$ where $\tilde q$ is cheaper to compute thant $q$. Then, for $(x,y)\in\R^d\times \R^d$, a random variable $V\sim p(x,v;\cdot)$ can be sampled as follows. Draw two independent random variables $U_1$ and $U_2$ uniformly distributed over $[0,1]$. If $U_1 \leqslant \tilde q(x,v)$ and $U_2 \leqslant q(x,v)/\tilde q(x,v)$, draw $V$ according to $p_1(x,v:\cdot)$ else draw $V$ according to $p_2(x,v;\cdot)$. That way, obviously, $V\sim p(x,v;\cdot)$. The trick is that if $U_1 > \tilde q(x,v)$ then we already know that $V$ has to be drawn according to $p_2(x,v;\cdot)$ and in that case we don't even have to compute $q(x,v)$. The smaller is $\tilde q$, the higher is the computational gain. 

We can go a bit further in two cases for which the first step $K$ where $q$ has to be computed, i.e. where $U_1 < \tilde q(X_K,V_K)$, can be computed more efficiently than with Bernoulli variables at each step.

\begin{itemize}
\item If $\tilde q(x,v) = \tilde q \in(0,1)$ is constant. In that case, $K+1$ follows a geometric law with parameter $\tilde q$ and can be sampled through the inverse transformation method i.e. by setting $K = \lfloor \ln U/\ln\tilde q \rfloor$ with $U$ uniformly distributed over $[0,1]$ (which is particularly more efficient than with Bernoulli variables when $\tilde q$ is small). 

\item  If $p_2(x,v;\cdot) = \delta_{f(x,v)}(\cdot)$ is deterministic, for some $f:\R^d\times\R^d \rightarrow \R^d$ (typically, for a persistent walk, $f(x,v) =v$). In that case, $K$ follows the distribution
\[\mathbb P \po K > n\pf \ = \ \prod_{k=0}^n \po 1 - \tilde q\po \varphi_k(x,v)\pf \pf \]
where $\varphi_0(x,v)=(x,v)$, $\varphi_{1}(x,v)=(x+\delta (v+f(x,v))/2,f(x,v))$ and $\varphi_{k+1}= \varphi_{k} \circ \varphi_1$  for all $k\in\N$. In particular cases, depending on $f$ and $\tilde q$, this distribution may again be sampled through the inverse transformation method.
\end{itemize}
In both cases, the algorithm is thus the following: draw $K$ as above and an independent $U_2$ uniformly distributed over $[0,1]$. Sample $(X_n, V_n)_{n\in\cco 0,K\ccf}$ as  a kinetic chain associated to the transition $p_2$. If $U_2 \leqslant q(X_{K},V_{K})/\tilde q(X_{K},V_{K})$, draw $V_{K+1}$ according to $p_1(X_{K},V_{K};\cdot )$, else according to $p_2(X_{K},V_{K};\cdot )$, and in both cases set $X_{K+1} = X_K+\delta(V_K+V_{K+1})/2$. Then, draw a new $K'$ in a similar way as $K$, etc.

In the general case, of course each of the kernels $p_1$ and $p_2$ may also be decomposed in a similar way as $p$, so that we end up for some $N\geqslant1$ with a decomposition $p=\sum_{i=1}^N p_n q_n$ where for each $n\in\cco 1,N\ccf$, $p_n$ is a transition kernel and $q_n\in[0,1]$ with some $\sum_{i=1}^N q_n = 1$. Similarly, if  $\tilde q(x,y) \leqslant \hat q(x,y)$ with  $\hat q(x,v)$ cheaper to compute than $\tilde q(x,v)$, then we can sample a Bernoulli variable with parameter $q(x,v)$ as the product of three Bernoulli variables with respective parameters $\hat q(x,v)$, $\tilde q(x,v) / \hat q(x,v)$ and $q(x,v)/\tilde q(x,v)$. In other words, in the general decomposition we can decompose each weight $q_n$ as a product $\prod_{k=1}^{r_n} q_{n,k}$ for some $r_n \geqslant 1$ and $q_{n,k} \in [0,1]$ for each $k\in\cco 1,r_n\ccf$. At the end of the day we get a representation of the form 
\[p(x,v;\cdot) \ = \ \sum_{n=1}^N p_n(x,v;\cdot) \prod_{k=1}^{r_n} q_{n,k}(x,v) \]
that we can use to sample according to $p$ in such a way that the average cost of computation is minimized. See Sections \ref{subSect-ZZd} and \ref{Section-Schema-Euler} for examples and related questions, in particular the link with factorization for Metropolis acceptance probabilities in Section \ref{SubSection-FactorisationZZZ}. See also \cite{MonmarcheBouncyChimie} for an application in molecular dynamics.

\subsection{Invariant measure}\label{SubsectionInvariantMeasure}

For MCMC applications, usual continuous-time kinetic Markov processes are designed to sample according to a given probability measure of the form $\mu(\dd x,\dd v) = \pi(\dd x) \otimes \nu(\dd v)$ on $\R^{d}\times\R^d$. The target is the position marginal $\pi$  and the velocity marginal $\nu$ can be chosen by the user, usual choices being Gaussian or uniform (over the sphere or a discrete set of velocities) distributions. By definition, $\mu $ is invariant for the kinetic walk on $\R^d$ associated to some kernel $p$ and time-step $\delta$ if 
\[\int f\po x + \frac{\delta}{2}(v+w),w\pf p(x,v;\dd w) \mu(\dd x,\dd v) \ = \ \int f\po x, v\pf \mu(\dd x,\dd v)\]
for all   $f\in\mathcal M_b(\R^{2d})$. 
%
%
Nevertheless, there are cases (such as Examples 1 and 2 of Section~\ref{SubSectionExKinetic}) where this condition is only approximately satisfied, typically $\mu$ is invariant only for the continuous-time process that is approximated by the discrete walk. In that case, for a fixed $\delta>0$, the kinetic walk, with timestep $\delta$ and associated to some transition $p_\delta$, typically admits some invariant measure $\mu_{\delta}$, which is also an invariant measure for the continuous-time Markov chain with generator
\[L_{\delta} f(x,v) \ = \ \int  \po f\po x + \frac{\delta}{2}(v+w),w\pf -f(x,v)\pf  p_{\delta}(x,v;\dd w) \,. \] 
In the typical case where $L_{\delta}$ converges,  as $\delta$ vanishes, to some $L$ for which $\mu$ is invariant, it is possible to obtain explicit error bounds between $\mu_{\delta}$ and $\mu$. We recall here a general argument based on Stein's method \cite[Section 6.2]{MattinglyStuartretyakov}.

\begin{prop}\label{Prop:invariant}
Let $L$ and $L_{\delta}$ be two Markov generators on some Polish space $E$ and let $\mu,\mu_\delta\in\mathcal P(E)$ be invariant measures of, respectively, $L$ and $L_\delta$. Denote $(P_t)_{t\geqslant 0}$ the semi-group associated to $L$. Suppose that there exist $C,\rho,h_{\delta}>0$ and two norms $N_1$ and $N_2$ on a subspace $\tilde {\mathcal M}$ of $M_b(E)$ such that the following holds:  1) $\tilde {\mathcal M}$ is dense in $(M_b(\R^{2d}),\|\cdot\|_\infty)$ and is contained by the domains of $L$ and $L_\delta$; 2)  for all $f\in \tilde {\mathcal M}$  and all $t\geqslant 0$, $N_1(P_t f-\mu f) \leqslant Ce^{-\rho t} N_2(f)/(1\vee \sqrt{t})$; 3) for all $f\in \tilde {\mathcal M}$, $\|(L-L_\delta) f\|_\infty \leqslant h_{\delta} N_1(f)$. Then, denoting $N_2(\mu-\mu_\delta)=\sup\{\mu(f) - \mu_\delta(f),\ N_2(f) \leqslant 1\}$, it holds:
\[N_2(\mu-\mu_\delta) \ \leqslant \ C(2/3+1/\rho)h_\delta\,.\]
\end{prop}
\begin{proof}
Let $f\in\tilde{\mathcal M}$ and $g=\int_0^\infty P_t (f-\mu f) \dd t$, which is well defined and satisfies  $N_1(g) \leqslant C(2/3+1/\rho)N_2(f)$ and solves the Poisson equation $Lg=\mu f-f$. Using that $\int L_\delta g \mu_{\delta} = 0$ by invariance of $\mu_\delta$, we get that
\[\mu(f) - \mu_\delta(f) \  = \ \mu_\delta(Lg) \ = \ \mu_\delta\po (L-L_\delta)g\pf \ \leqslant \ h_\delta  C(2/3+1/\rho)N_2(f)\,.\]
\end{proof}
Remark that, in Proposition \ref{Prop:invariant}, $C$ and $\rho$ only depend on the limit process $L$. For particular processes, their existence usually follows from  regularization and  ergodicity results for $L$ (with $N_1$ and $N_2$, typically, $L^2(\mu)$, $H^1(\mu)$ or $V$-norms associated to some Lyapunov function), see e.g.  Section \ref{SubSectionTCLZZZd}, \cite{GlynnMeyn,MattinglyStuartretyakov} or, for velocity jump processes, \cite[Section 3]{Monmarche2017}. Then, if we are given a family of generators $L_\delta$ for all $\delta>0$ such that $\|(L-L_\delta) f\|_\infty \leqslant h_{\delta} N_1(f)$ with $h_\delta\rightarrow 0$ for $\delta\rightarrow 0$, we get a quantitative estimate on the convergence $\mu_\delta\rightarrow \mu$ (remark that, though the uniqueness of the invariant measure is ensured for $L$ from the geometric ergodicity assumption, we haven't assumed the uniqueness of the invariant measure for $L_\delta$).

In fact, reiterating this argument, following the Talay-Tubaro method \cite{TalayTubaro},  an expansion of the bias $\mu_\delta f -  \mu f$ in term of powers of $\delta$ can be computed and a Romberg extrapolation (or related methods) can be used to kill the first order terms, see  \cite[Section 2.3]{TalayTubaro} and \cite[Section 3.3.4]{LelievreStoltz}. 

Proposition \ref{Prop:invariant} can be shown to apply in Example 2 of Section~\ref{SubSectionExKinetic}, i.e. the Langevin dynamics, see for instance \cite{LeimkuhlerMatthewsStoltz}. On the contrary, it does not apply to the Hamiltonian dynamics of Example 1. Indeed, in that case, the limit process admits many invariant measures (because of energy conservation) and thus $C$ and $\rho$ cannot exist.


With Proposition \ref{Prop:invariant}, we have seen that the convergence of the invariant measures is related to the convergence of the generators. Now, the latter is classically related to the convergence of the processes and we address this question in the next section. See Theorem \ref{Thm-Kallenberg} below for some considerations on the convergence of the generators in the specific case of kinetic walks.

\subsection{Scaling limits}\label{Section-scaling-general}

In this section, we consider for all $\varepsilon\in(0,1]$ a kinetic walk  $(X_n^\varepsilon,V_n^\varepsilon)_{n\in\N}$ on $\R^d\times \R^d $  with timestep $1$, kernel $p_\varepsilon$ and  initial condition $(x_0^\varepsilon,v_0^\varepsilon)$. We are interested in the possible convergence of this chain, possibly rescaled as $\varepsilon$ vanishes, toward a continuous-time process. The regime for which $(X_n^\varepsilon)_{n\in\N}$ converges toward an elliptic diffusion has been abundantly studied for second-order chains, see \cite[Section 5]{GruberThese} and references therein, and for this reason we will mostly focus on the cases where the full system $(X_n^\varepsilon,V_n^\varepsilon)_{n\in\N}$ converges toward a continuous-time kinetic process $(Y_t,W_t)_{t\geqslant 0}$, where kinetic means that $Y_t = Y_0+\int_0^t W_s\dd s$.

To alleviate notations, unless otherwise specified, we drop the $\varepsilon$ superscript in all the rest of the section and simply write $(X_n,V_n)$. We start with an informal discussion on the scaling in the simple case where the dynamics are homogeneous with respect to the space variable $x$. Since, in order to expect a continuous-time limit, $X_n$ should be nearly constant over a large number of steps as $\varepsilon$ goes to zero, this homogeneous case  should be expected to describe the short time dynamics of the general case.

\subsubsection{The space homogeneous case}

Consider the case where $p_\varepsilon(x,v;\cdot) = h_\varepsilon(v;\cdot)$ for some transition kernel $h_\varepsilon : \R^d  \rightarrow \R^d$. In that case, $(V_n)_{n\in\N}$ is a Markov chain by itself. Since
\begin{eqnarray}\label{Eq-space-homogeneous}
X_n &=& X_0 +  \sum_{k=1}^n V_k + \frac{V_0-V_n}{2}\,,
\end{eqnarray}
then, provided that, say, the variance of the last term is bounded uniformly in $n$, the situation is essentially the same as the case of correlated random walks. In term of time/space scaling, different cases may be distinguished concerning the limit of $(X_n)_{n\in\N}$:
\begin{itemize}
\item If $(V_k)_{k\in\N}$ is in fact an i.i.d. sequence -- namely if $v\mapsto h_\varepsilon(v,\cdot)$ is  constant -- then, up to a vanishing term,  $(X_n)_{n\in\N}$ is a simple random walk. If it admits a continuous-time limit, then the latter is necessarily a Levy process, and conversely any Levy process $(L_t)_{t\geqslant 0}$ may be obtained as a scaling limit of such a walk, even if we restrict the question to walks on $\Z^d$: indeed,  considering $V_k = \lfloor (L_{(k+1)\varepsilon} - L_{k\varepsilon})/\varepsilon\rfloor$, then $(\varepsilon X_{\lfloor t/\varepsilon\rfloor})_{t\geqslant 0} \rightarrow (L_t)_{t\geqslant 0}$.

Let us consider two particular cases. First, suppose that there exist $\eta:(0,1]\rightarrow (0,+\infty)$ such that 
\begin{eqnarray*}
\frac{\eta(\varepsilon)}{\varepsilon} \mathbb E (V_1) & \underset{\varepsilon\rightarrow 0}\longrightarrow & \mu \in\R^d\\
\frac{\eta^2(\varepsilon)}{\varepsilon} \text{Var}(V_1) & \underset{\varepsilon\rightarrow 0}\longrightarrow & \Sigma \in \mathcal M^{sym\geqslant 0}_{d\times d}(\R)
\end{eqnarray*}
(with possibly $\mu=0$ or $\Sigma = 0$). Then, from Donsker's Theorem, provided that $\eta(\varepsilon) x_0^\varepsilon$ converges to some $x_0^*\in\R^d$, we get the drifted Brownian motion
\[(\eta(\varepsilon)X_{\lfloor t/\varepsilon\rfloor })_{t\geqslant 0} \ \overset{law} {\underset{\varepsilon\rightarrow 0}\longrightarrow}\ (x_0^*+ t\mu + \Sigma^{\frac12} B_t)_{t\geqslant 0}\,,\]
where $(B_t)_{t\geqslant 0}$ is a standard Brownian motion. Second, suppose that there exist $\lambda >0$, $\mu\in\R^d$, $\nu\in\mathcal P(\R^d)$, $\eta:(0,1]\rightarrow(0,+\infty)$ and $\mu:(0,1] \rightarrow\R^d$ such that
\begin{eqnarray*}
\mathbb P \po  V_1 = \mu(\varepsilon)\pf & = & 1 - \lambda \varepsilon + \underset{\varepsilon\rightarrow 0}o(\varepsilon)\\
\mu(\varepsilon) & \underset{\varepsilon\rightarrow 0}\longrightarrow & \mu\\
\mathcal Law\po  \eta(\varepsilon) V_1\ | \ V_1\neq \mu(\varepsilon) \pf & \underset{\varepsilon\rightarrow 0}\longrightarrow & \nu \,.
\end{eqnarray*}
Then the cardinality $N_t^\varepsilon$ of $\{n\in\N\ : \varepsilon n\leqslant t,\ V_k \neq \mu(\varepsilon)\}$ converges as $\varepsilon$ vanishes to a Poisson process $N_t$ with intensity $\lambda$. As a consequence, considering an i.i.d. sequence $(W_k)_{k\in\N}$ with law $\nu$ independent from $(N_t^\varepsilon)_{t\geqslant 0}$ for all $\varepsilon$ and from $(N_t)_{t\geqslant 0}$, provided that $\eta(\varepsilon) x_0^\varepsilon$ converges to some $x_0^*\in\R^d$, we get the drifted compound Poisson process
\[(\eta(\varepsilon) X_{\lfloor t/\varepsilon\rfloor })_{t\geqslant 0}   \ \overset{law} { \underset{\varepsilon\rightarrow 0}\longrightarrow}\ x_0^* + t \mu +  \po \sum_{n=0}^{N_t  } W_k\pf_{t\geqslant 0} \,.\]
\end{itemize}

In the more general situation where $(V_k)_{k\in\N}$ is not necessarily an i.i.d. sequence, if there exists a $\kappa:(0,1]\rightarrow (0,+\infty)$ such that $\po \kappa(\varepsilon) V_{\lfloor t/\varepsilon\rfloor}\pf_{t\geqslant 0} $ converges in distribution toward a continuous-time Markov process $(W_t)_{t\geqslant 0}$ then, denoting $\eta(\varepsilon) = \varepsilon \kappa(\varepsilon)$, \eqref{Eq-space-homogeneous} reads
\[\eta(\varepsilon) X_{\lfloor t/\varepsilon \rfloor}  \ = \ \eta(\varepsilon) x_0^\varepsilon + \int_0^t \kappa(\varepsilon) V_{\lfloor s/\varepsilon\rfloor} \dd s + \frac{ \varepsilon}2\po \kappa(\varepsilon)  V_{\lfloor t/\varepsilon \rfloor}  - \kappa(\varepsilon)  V_0\pf \,.\]
Integration being continuous with respect to Skorohod convergence, provided that $\eta(\varepsilon) x_0^\varepsilon\rightarrow x_0^*$ as $\varepsilon$ vanishes, this yields
\[\po \eta(\varepsilon) X_{\lfloor t/\varepsilon \rfloor},\kappa(\varepsilon) V_{\lfloor t/\varepsilon\rfloor}\pf_{t\geqslant 0} \ \underset{\varepsilon\rightarrow 0}\longrightarrow\  \po x_0^* + \int_0^t W_s \dd s\,,\ W_t\pf_{t\geqslant 0}\,. \]
Remark that, of course, the scaling factor for the space variable is fixed by the scaling factors of the time and velocity variables.
\begin{itemize}
\item For instance, if $(V_k)_{k\in\N}$ is a random walk on $\Z^d$, then as seen before it can converge toward a drifted Brownian motion, in which case the scaling limit of $(X_n,V_n)_{n\in\N}$ is the  Langevin diffusion, i.e. the solution of the SDE
\[\left\{\begin{array}{rcl}
\dd X_t &=& V_t \dd t\\
\dd V_t  &= & \mu + \Sigma^{1/2}\dd B_t\,,
\end{array} \right.\]
where $(B_t)_{t\geqslant 0}$ is a standard Brownian motion on $\R^d$.

\item Alternatively, if there exist $\lambda >0$, $\nu\in\mathcal P(\R^d)$ and $\kappa:(0,1]\rightarrow(0,+\infty)$ such that
\begin{eqnarray*}
\mathbb P \po  V_{k+1} = V_k\pf & = & 1 - \lambda \varepsilon + \underset{\varepsilon\rightarrow 0}o(\varepsilon)\\
\mathcal Law\po  \kappa(\varepsilon) V_{k+1}\ | \ V_{k+1}\neq V_k\pf & \underset{\varepsilon\rightarrow 0}\longrightarrow & \nu \,,
\end{eqnarray*}
then $(\kappa(\varepsilon) V_{\lfloor t/\varepsilon \rfloor})_{t\geqslant 0}$ converges as $\varepsilon$ vanishes toward $(Y_{N_t})_{t\geqslant 0}$, 
where $(Y_k)_{k\in\N}$ is an i.i.d. sequence with law $\nu$ and $(N_t)_{t\geqslant 0}$ is a Poisson process with intensity $\lambda$, independent from $(Y_k)_{k\in\N}$. In that case, $(X_t,V_t)_{t\geqslant 0}$ is the velocity jump process associated to the linear Boltzmann (or BGK) equation \cite{Bouin}.
\end{itemize}

These different examples highlighted three cases: if there is no inertia, the velocity tends to mix fast and $(X_n)_{n\in\N}$ is Markovian. If there is some inertia in the sense that the velocity tends to stay aligned from one step to the other but with possible small fluctuations, the limit is a kinetic diffusion. If the velocity is rigorously constant for large times, the chain converges toward a velocity jump process. Of course this is  a  non-exhaustive list: in general, drift, diffusion, Poisson or $\alpha$-stable jumps can all be present in the limit, either kinetic or not (see in particular Section \ref{Section-Schema-Euler}). But with these three regimes we cover the cases of the processes classically used in MCMC sampling. 


\subsubsection{The case of kinetic limits}

We keep the notations of the beginning of Section \ref{Section-scaling-general}. In particular, for all $\varepsilon\in(0,1]$, $(X_n^\varepsilon,V_n^\varepsilon)_{n\in\N}$ is a kinetic walk on $\R^d\times \R^d$  with transition $p_\varepsilon$ and   time-step $1$. Let $\kappa:(0,1]\rightarrow (0,+\infty)$ and, for all $\varepsilon\in(0,1]$, set $\eta(\varepsilon) = \varepsilon \kappa(\varepsilon)$. For any fixed $x\in\R^d$, let $(\tilde V_k^{\varepsilon,x})_{k\in\N}$ be the Markov chain on $\R^d$ with transitions 
\[\tilde V_{k+1}^{\varepsilon,x}/\kappa(\varepsilon)\ \sim \ p_\varepsilon(x/\eta(\varepsilon),\tilde V_k^{\varepsilon,x}/\kappa(\varepsilon);\cdot)\,,\]
and let $N_t^\varepsilon$ be a Poisson process with intensity $1/\varepsilon$. Denote  $\tilde  L_{\varepsilon}$ the operator defined on $\mathcal M_b(\R^{2d})$ by
\[\tilde  L_{\varepsilon} f(x,v) \ = \ \frac1\varepsilon\int\co f(x,\kappa(\varepsilon) w) - f(x,v) \cf p_\varepsilon(x/\eta(\varepsilon),v/\kappa(\varepsilon);\dd w)
\,,\]
which is the infinitesimal generator of the Feller process $(\tilde X_{t}^{x},\tilde V_{N_t^\varepsilon}^{\varepsilon,x})_{t\geqslant 0}$, where we simply set $\tilde X_{t}^{x}=x$ for all $t\geqslant 0$ (hence the link with the space homogeneous case).  A direct corollary of \cite[Theorem  17.28]{Kallenberg} is the following:

\begin{thm}\label{Thm-Kallenberg}
Suppose that there exists a Feller generator $\tilde L$ on $\R^{2d}$ with domain containing $\mathcal C^2_c(\R^{2d})$ and such that $\|\tilde  L_\varepsilon f-\tilde Lf\|_\infty \rightarrow 0$ as $\varepsilon \rightarrow 0$ for all $f\in\mathcal C^2_c(\R^{2d})$. Define the operator $L$ by
\[Lf(x,v) \ = \ v\cdot \na_x f(x,v) + \tilde L f(x,v)\,.\]
Suppose that $L$ is the infinitesimal generator of a Feller process $(Y_t,W_t)_{t\geqslant 0}$ and that $\mathcal C^2_c(\R^{2d})$ is a core of $L$. Suppose that for all compact set $\mathcal K\subset \R^d$,
\begin{eqnarray}\label{Eq-Scaling-vitesse}
\sup_{x\in\mathcal K}\sup_{v\in\R^d} \int   | v -\kappa(\varepsilon)w|^2    p_\varepsilon\po \frac{x}{\eta(\varepsilon)},\frac{v}{\kappa(\varepsilon)};\dd w\pf & \underset{\varepsilon\rightarrow 0}\longrightarrow   & 0\,.
\end{eqnarray}
Finally, suppose that $(\eta(\varepsilon)X_0^\varepsilon,\kappa(\varepsilon) V_0^\varepsilon)$ converges in law toward $(Y_0,W_0)$ as $\varepsilon$ vanishes. Then
\[\po \eta(\varepsilon)  X^\varepsilon_{\lfloor t/\varepsilon\rfloor},\kappa(\varepsilon)  V^\varepsilon_{\lfloor t/\varepsilon\rfloor}\pf_{t\geqslant 0} \ \overset{law}{\underset{\varepsilon\rightarrow 0}\longrightarrow}\  (Y_t,W_t)_{t\geqslant 0}\,.\]
\end{thm}

See some applications with Proposition \ref{Thm-Dimd-scaling},  Section \ref{SubSectionStrang} or \cite{MonmarcheBouncyChimie}.

\begin{proof}
Denoting $(Y_t^\varepsilon,W_t^\varepsilon) = (\eta(\varepsilon) X_{N_t^\varepsilon}^\varepsilon,\kappa(\varepsilon) V_{N_t^\varepsilon}^\varepsilon)$,  the generator $L_\varepsilon$ of $(Y_t^\varepsilon,W_t^\varepsilon)_{t\geqslant 0}$ is defined on $\mathcal M_b(\R^{2d})$ by
\[ L_{\varepsilon} f(x,v) \ = \ \frac1\varepsilon\int\co f(x+\varepsilon (v+w)/2,\kappa(\varepsilon) w) - f(x,v) \cf p_\varepsilon(x/\eta(\varepsilon),v/\kappa(\varepsilon);\dd w)\,.\]
Note that $L_{\varepsilon}f \in \mathcal M_b(R^{2d})$ for all $f \in \mathcal M_b(R^{2d})$, so that if $(X_0^\varepsilon,V_0^\varepsilon)=(x,v)$,
\begin{eqnarray*}
\mathbb E\po f\po Y_t^\varepsilon,W_t^\varepsilon \pf \pf & = & \sum_{k\in\N} \mathbb P\po N_t^\varepsilon = k\pf \mathbb E\po   f\po  \eta(\varepsilon) X_{k}^\varepsilon,\kappa(\varepsilon) V_{k}^\varepsilon\pf  \pf \\
&= &f(x,v) + t L_{\varepsilon}f(x,v) + \underset{t\rightarrow 0}o(t)\,,
\end{eqnarray*}
with a negligible term uniform in $(x,v)$ (which will be the case of all the  negligible terms in the rest of the proof). This means that $\mathcal M_b(\R^{2d})$, hence $\mathcal C^2_c(\R^{2d})$, is included in the strong domain of $L_{\varepsilon}$ for all $\varepsilon\in(0,1]$. From \cite[Theorem 17.28]{Kallenberg} and the assumption that $\mathcal C^2_c(\R^{2d})$ is a core for $L$, it only remains to check that $\|L_\varepsilon f- Lf\|_\infty $ vanishes with $\varepsilon$ for all $f\in\mathcal C^2_c(\R^{2d})$. Now, indeed, for $f\in\mathcal C^2_c(\R^{2d})$,
\begin{eqnarray*}
L_{\varepsilon} f(x,v) &=& \tilde  L_{\varepsilon} f(x,v)\\
& &   + \ \frac1\varepsilon\int\co f\po x+\frac{\varepsilon}2 (v+\kappa(\varepsilon) w),\kappa(\varepsilon) w\pf  - f(x,\kappa(\varepsilon) w) \cf p_\varepsilon\po \frac{x}{\eta(\varepsilon)},\frac{v}{\kappa(\varepsilon)};\dd w\pf \\
 &=&  \tilde L f(x,v) + \int \frac{v+\kappa(\varepsilon) w}2 \cdot \na_x f\po x ,\kappa(\varepsilon) w\pf   p_\varepsilon\po \frac{x}{\eta(\varepsilon)},\frac{v}{\kappa(\varepsilon)};\dd w\pf  + \underset{t\rightarrow 0}o(1) \,.
\end{eqnarray*}
Considering a ball $\mathcal B\subset\R^d$ of some radius $R$ such that the support of $f$ is included in $\mathcal B^2$, we bound
\begin{eqnarray*}
 \left|  \po v-\frac{v+\kappa(\varepsilon) w}2\pf \cdot \na_x f\po x ,\kappa(\varepsilon) w\pf   \right| & \leqslant & \frac12 \|\na_x f\|_\infty  |v-\kappa(\varepsilon) w| \1_{\mathcal  B}(x) \1_{\mathcal  B}(\kappa(\varepsilon) w) \\
|\na_x f\po x ,\kappa(\varepsilon) w\pf - \na_x f\po x ,v\pf|  & \leqslant & | v -\kappa(\varepsilon)w| \| \na^2 f\|_\infty \1_{\mathcal B}(x)\po \1_{\mathcal B}(v) + \1_{\mathcal B}(\kappa(\varepsilon) w)\pf \\
|v| \po \1_{\mathcal B}(v) + \1_{\mathcal B}(\kappa(\varepsilon) w)\pf& \leqslant & 2R + | v -\kappa(\varepsilon)w|\,.
\end{eqnarray*}
As a consequence, for some $C_f>0$, for all $(x,v)\in\R^{2d}$,
\begin{multline*}
| L_{\varepsilon} f(x,v) - L f(x,v)| \ \leqslant \\ \1_{\mathcal B}(x) C_f \int  \po | v -\kappa(\varepsilon)w| + | v -\kappa(\varepsilon)w|^2\pf    p_\varepsilon\po \frac{x}{\eta(\varepsilon)},\frac{v}{\kappa(\varepsilon)};\dd w\pf   + \underset{t\rightarrow 0}o(1) \,.
\end{multline*}
Condition \eqref{Eq-Scaling-vitesse} concludes. 
\end{proof}

\section{The discrete Zig-Zag walk}\label{subSect-ZZd} 

This section is devoted to the definition and study of a discrete-space analogous of the Zig-Zag process on $\R^d$. 

\subsection{Definition}

Let $d\in\N_*$, $U:\Z^d \rightarrow \R$ be such that $\mathcal Z = \sum_{x\in\Z^d} \exp(-U(x)) < +\infty$,  $\pi(x) = \exp(-U(x))/\mathcal Z$ be the associated Gibbs distribution and $\mu(x,v) = \pi(x)/{2^d}$ for $v\in\{-1,1\}^d$ and $x\in\Z^d$.  For $i\in\cco 1,d\ccf$, denote $e_i$ the $i^{th}$ vector of the canonical basis of $\R^d$ and let
\begin{eqnarray}\label{Eq-ZZZd-transition-p}
q_i(x,v_i) & = & \min\po \frac{\pi(x+v_ie_i)}{\pi(x)},1\pf \ = \ e^{-\po U(x+v_ie_i) - U(x)\pf_+}\notag \\
p_i(x,v_i;w_i) &=& q_i(x,v_i) \delta_{v_i}(w_i) + (1-q_i(x,v_i)) \delta_{-v_i}(w_i)\notag \\
p(x,v;w) &= & \prod_{i=1}^d p_i\po x + \sum_{j=1}^{i-1} \frac{v_j+w_j}2 e_j, v_i,w_i\pf  \,.
\end{eqnarray}
That way, $p:\Z^d\times \{-1,1\}^d\rightarrow\mathcal P(\{-1,1\}^d)$. We call Zig-Zag walk on $\Z^d$ the kinetic walk $(X_n,V_n)_{n\in\N}$ associated to this kernel $p$ with timestep $\delta=1$, i.e. the Markov chain on $\Z^d\times\{-1,1\}^d$ whose transitions are given by \eqref{Eq-transition-kinetic}. Remark that, for $d=1$, we retrieve the chain studied in Section \ref{Section-Dim1}.

A random variable $V\sim p(x,v;\cdot)$ in $\{-1,1\}^d$ can be sampled as follows. Set $Y_0=x$, and suppose by induction that $Y_{k-1}\in\Z^d$ has been defined for some $k\in\cco 1,d\ccf$.  Set $V_{k} = v_{k}$ with probability $q_k(Y_{k-1},v_k)$ and $V_{k}=-v_{k}$ else, and in either case set $Y_k = Y_{k-1}+(v_k+V_k)/2$. Then $V$ is distributed according to $p(x,v;\cdot)$, and $X:=Y_d = x+(v+V)/2$. In other words, this is a Gibbs algorithm based on the Zig-Zag walk on $\Z$: one step of the Zig-Zag walk in $\Z^d$ is the result of $d$ successive one-dimensional Zig-Zag steps on each coordinate, the others being fixed.

If we want the coordinates to play  a symmetric role in the transition, for $\sigma$ a permutation of $\cco 1,d\ccf$ we can define $p_\sigma(x,v;\cdot)$ to be the law of $W_{\sigma^{-1}}$ when $W\sim p(x_{\sigma},v_{\sigma};\cdot)$, where $u_\sigma$ for $u\in\Z^d$ and $\sigma\in\mathfrak{S}_d$ denotes $(u_{\sigma(1)},\dots,u_{\sigma(d)})$. This accounts to use the order given by $\sigma$ to update the coordinates. Then
\[p_{sym}(x,v;\cdot) \ = \ \frac{1}{d!} \sum_{\sigma \in \Sigma_d} p_\sigma(x,v;\cdot)\]
corresponds to a transition where the order is sampled at random at each step of the Zig-Zag walk. There is no particular practical interest to consider $p_{sym}$ rather than $p$, and moreover any result on $p$ can straightforwardly be adapted to $p_\sigma$ by renumbering of the coordinates, and then to $p_{sym}$.

\subsection{Equilibrium and scaling limit}

\begin{prop}\label{Prop-ZZZd-invariance}
The probability distribution $\mu$ is invariant for the Zig-Zag walk on $\Z^d$.
\end{prop}

\begin{proof}
As proven in Section \ref{Section-Dim1}, for all fixed $(x_j,v_j)_{j\in\cco 2,d\ccf}$, the transition  on $\Z\times\{-1,1\}$ defined by
\[W_{n+1} \sim p_1\po [Y_n,x_2,\dots,x_d],[V_n,v_2,\dots,v_d];\cdot\pf\,,\qquad Y_{n+1} = \frac{W_n+W_{n+1}}{2}\]
admits the conditional law $(y,w)\mapsto \pi(y,x_2,\dots,x_d)/2$ as an invariant measure. As a consequence, the transition of  the Markov chain $(\tilde X_n,\tilde V_n)$ on $\Z^d\times\{-1,1\}^d$ with $(\tilde X_{n,1},\tilde V_{n,1})=(Y_n,W_n)$ and $(\tilde X_{n,j},\tilde V_{n,j})=(\tilde X_{0,j},\tilde V_{0,j})$ for $j\neq 1$ also fixes $\mu$. Since the transition of the Zig-Zag walk is the composition of $d$ such transitions, it fixes $\mu$.
\end{proof}

Remark that, if the target law is of a tensor form $\pi(x) = \prod_{i=1}^d \pi_i(x_i)$, then the coordinates of a Zig-Zag walk are just $d$ independent one-dimensional Zig-Zag walks (which is similar to the continuous-time process).

Recall that the continuous-time Zig-Zag process on $\R^d$ associated to a potential $H$ is the Markov process on $\R^d\times\{-1,1\}^d$ with generator
\[Lf(x,v) \ = \ v\cdot \na_x f(x,v) + \sum_{i=1}^d (v_i\partial_{x_i} H(x))_+ \po f\po x,v_{-i}\pf - f(x,v)\pf\,,\]
where we denote by $v_{-i}$ the vector of $\{-1,1\}^d$ obtained from $v$ by multiplying its $i^{th}$ coordinate by $-1$. The following is the extension of Theorem \ref{Thm-Dim1-scaling} in larger dimension. 

\begin{thm}\label{Thm-Dimd-scaling}
For $H\in\mathcal C^2(\R^d)$ that goes to infinity at infinity, for all $\varepsilon>0$, define $U_\varepsilon : \Z^d\mapsto \R$ by $U_\varepsilon(x) = H(\varepsilon x)$ for all $x\in\Z^d$. Let $(X^\varepsilon_k,V^\varepsilon_k)_{k\in\N}$ be the Zig-Zag walk on $\Z^d$ associated to $U_\varepsilon$ and with some initial condition $(x_0^\varepsilon,v_0)$. Suppose that $\varepsilon x_0^\varepsilon$ converges to some $x_0^*\in\R^d$ as $\varepsilon$ vanishes. Then 
\[\po \varepsilon X^\varepsilon_{\lfloor t/\varepsilon\rfloor},V^\varepsilon_{\lfloor t/\varepsilon\rfloor}\pf_{t\geqslant 0} \ \overset{law} {\underset{\varepsilon \rightarrow 0}\longrightarrow}\ \po Y_{t},W_{t}\pf_{t\geqslant 0}\,,\]
where $(Y_t,W_t)_{t\geqslant 0}$ is a  Zig-Zag process on $\R^d$ associated to $H$ and with $(Y_0,W_0)=(x_0^*,v_0)$.
\end{thm}

\begin{proof}
Let us show that Theorem \ref{Thm-Kallenberg}, or rather directly \cite[Theorem  17.28]{Kallenberg}, applies. First, following \cite{DurmusGuillinMonmarche2018}, we can see that the continuous-time process can be smoothly and compactly approximated (in the sense of \cite[Definition 20]{DurmusGuillinMonmarcheToolbox}) by replacing its continuous jump rates by $\mathcal C^\infty$ jump rates (see the case of the BPS in \cite[Proposition 23]{DurmusGuillinMonmarcheToolbox} for details). From \cite[Theorem 21]{DurmusGuillinMonmarcheToolbox}, this proves that $\mathcal C^1_c(\R^d\times\{-1,1\}^d)$ is a core for the strong generator $L$ of the Zig-Zag process.
 Denote $L_\varepsilon$ the generator of $(\varepsilon X^\varepsilon_{N_t},V^\varepsilon_{N_t})_{k\in\N}$ where $(N_t)_{t\geqslant 0}$ is a Poisson process with intensity $1/\varepsilon$. Then all bounded measurable functions $f$ are in the domain of $L_\varepsilon$ and
\[L_\varepsilon f(x,v) \ = \ \frac{1}{\varepsilon}  \sum_{w\in\{-1,1\}^d} \po f\po   x+\varepsilon \frac{v+w}{2},w\pf  - f(x,v)\pf p_\varepsilon\po \frac x\varepsilon,v;w\pf\,,\]
 where $p_\varepsilon$ is given by \eqref{Eq-ZZZd-transition-p} with $U=U_\varepsilon$. Next, for all $i\in\cco 1,d\ccf$,
 \[\exp\po -\po U(x+\varepsilon v_i e_i) - U(x)\pf_+\pf \ = \ 1 - \varepsilon \po v_i\partial_{x_i}U(x)\pf_+ + \underset{\varepsilon\rightarrow 0}o(\varepsilon)\,,\]
 where the negligible term is uniform over all compact set of $\R^d\times\{-1,1\}^d$ since $H$ is $\mathcal C^2$ (and this will be the case for all the negligible terms below).  More generally, 
  \[\exp \po -\po U\po x+\varepsilon \sum_{j=1}^{i}v_j e_j\pf - U\po x+\varepsilon \sum_{j=1}^{i-1}v_j e_j\pf\pf_+\pf  \ = \ 1 - \varepsilon \po v_i\partial_{x_i}U(x)\pf_+ + \underset{\varepsilon\rightarrow 0}o(\varepsilon)\,,\]
  from which 
  \begin{eqnarray*}
  p_\varepsilon\po \frac x\varepsilon,v;w\pf & = &  \prod_{i=1}^d \po \delta_{v_i}(w_i) + \varepsilon \po v_i\partial_{x_i}U(x)\pf_+ \po \delta_{-v_i}(w_i) - \delta_{v_i}(w_i)\pf + \underset{\varepsilon\rightarrow 0}o(\varepsilon)\pf    \\
  & = & \delta_v(w) + \varepsilon \sum_{i=1}^d \po v_i\partial_{x_i}U(x)\pf_+ \po \delta_{v-2v_ie_i}(w) - \delta_{v}(w)\pf + \underset{\varepsilon\rightarrow 0}o(\varepsilon) \,.
  \end{eqnarray*}
  On the other hand, if $f\in\mathcal C^2_c(\R^d\times\{-1,1\}^d)$,
  \[f\po   x+\varepsilon \frac{v+w}{2},w\pf \ = \ f\po x,w\pf + \varepsilon  \frac{v+w}{2}\ \cdot \na_x f\po x,w \pf + \underset{\varepsilon\rightarrow 0}o(\varepsilon) \,, \]
so that
\[ \| L_\varepsilon f - L f\|_\infty \ \underset{\varepsilon\rightarrow 0}\longrightarrow \ 0\,,\]
and  Theorem \cite[Theorem  17.28]{Kallenberg}  concludes.
\end{proof}

Note that the space/time scaling in Theorem \ref{Thm-Dimd-scaling} is ballistic. It means in particular that, in $n=\lfloor 1/\varepsilon\rfloor$ steps, the Zig-Zag walk with potential $U_\varepsilon$ is at distance of order $n$ (and not $\sqrt n$ as in the diffusive case) from its starting point.

\subsection{Thinning and factorization}\label{SubSection-FactorisationZZZ}

In order to sample the Zig-Zag walk, a priori, at each time step, $U(x)$ has to be computed for $d+1$ values of $x$. However, thanks to the thinning method recalled in Section \ref{Section-thinning}, this computational cost may drop if simple bounds are known on the increments of $U$. 

Moreover, the factorization principle used for the continuous-time Zig-Zag process in \cite{BierkensFearnheadRoberts} to do subsampling is still available here. Suppose that we can decompose $U(x+v) - U(x) = \sum_{j=1}^M f_j(x,v)$ for all $(x,v)\in\Z^d\times\{-1,1\}^d$ for some $M\in\N_*$ and $f_j:\Z^d\times\{-1,1\}^d\rightarrow\R$. This is for instance the case if $U = \sum_{j=1}^M U_j$, in which case we can take $f_j(x,v) = U_j(x+v) - U_j(x)$, but in general the $f_j$'s are not required to be discrete gradients. Consider the Zig-Zag walk as defined above except that the probability $q_i(x,v_i)$ is replaced by
\[\tilde q_i(x,v_i) \ = \  \prod_{j=1}^M e^{-\po f_j(x,v_ie_i)\pf_+}\,.\]

\begin{prop}
This Zig-Zag walk with $\tilde q_i$ still admits $\mu$ as an invariant measure.
\end{prop}
\begin{proof}
As in the proof of Proposition \ref{Prop-ZZZd-invariance}, we just have to prove the result for $d=1$. Following Section \ref{Section-Dim1},  this stems from the same result applied to classical Metropolis-Hastings algorithms. Indeed, let $q$ be a symmetric Markov kernel on a space $E$ and let $\alpha:E\times E\rightarrow [0,1]$ be the acceptance probability of a Metropolis-Hastings  chain with proposal $q$, namely a chain with transition kernel $p(x,y)=q(x,y)\alpha(x,y)$ for $y\neq x$. Then this chain is reversible with respect to a probability $\mu$ if and only if
\[\frac{\alpha(x,y)}{\alpha(y,x)} \ = \ \frac{\mu(x)}{\mu(y)}\qquad \forall x,y\in E\,.\]
In particular, if $\mu(x) = C\prod_{i=1}^M g_i(x)$ for some positive $g_i$'s and a normalization constant $C$ then, setting $\alpha_i(x,y) = \min\po 1,g_i(x)/g_i(y)\pf$ for all $i\in\cco 1,M\ccf$ and $x,y\in E$ ensures that
\[\frac{\alpha_i(x,y)}{\alpha_i(y,x)} \ = \ \frac{g_i(x)}{g_i(y)}\qquad \forall i\in\cco 1,M\ccf\,, x,y\in E\,.\]
Taking the product over $i\in\cco 1,M\ccf$ and noting that $\min(1,a/b) = \exp\po -(\ln(a/b))_+\pf$ for $a,b>0$ concludes.
\end{proof}

The bad side of factorization is that it increases the number of rejections (for the Metropolis-Hastings algorithm, hence of collisions for the Zig-Zag process). On the other hand, if $U(x+v)-U(x)$ can be decomposed in a part that is cheap to compute and a part that may be expensive to compute but  is small and for which an efficient bound is available for thinning, then it may give a significant computational gain. This is particularly well-adapted for multi-scale potentials, as we can see in Section \ref{Section-Schema-Euler} on a similar problem.

\subsection{Irreducibility, Ergodicity, CLT}

\subsubsection{Irreducibility}

Irreducibility is a delicate question for the continuous Zig-Zag process, see \cite{BierkensRobertsZitt}. Here, for the discrete Zig-Zag walk,  we will only tackle the restrictive case where, following the definitions of \cite{BierkensRobertsZitt},  all velocities are asymptotically flippable\footnote{The nice proof of irreducibility of \cite{BierkensRobertsZitt} under weaker conditions on $U$ may possibly be partially adapted    for the discrete Zig-Zag walk. Nevertheless, note that the smoothness condition on $U$ has no discrete counterpart. In any case, the potential $U(x)=\|x\|_\infty$ would still be a counter-example for which the conclusion of  Proposition~\ref{Prop-ZZZd-Irreductible} would not hold.} and thus the proof is similar to the Gaussian case with dominant diagonal of \cite[Corollary 1]{BierkensRobertsZitt}. Anyway we are interested in the exponentially fast convergence toward equilibrium under the assumption of Proposition \ref{Prop-ZZZd-Lyapunov} below, which is even stronger. Moreover, with this restriction, we can focus on the specificities of the discrete realm.

Indeed, let us call $\sigma(x,v):=\po (-1)^{x_i}v_i\pf_{i\in\cco 1,d\ccf}\in\{-1,1\}^d$ the signature of $(x,v)\in\Z^d\times\{-1,1\}^d$. If $(X_n,V_n)_{n\in\N}$ is a Zig-Zag walk on $\Z^d$ then, like in the one-dimensional case, $\sigma(X_n,V_n)=-\sigma(X_{n-1},V_{n-1})$. In particular, denoting $\mathcal A_s = \{(x,v)\in\Z^d\times\{-1,1\}^d\ : \ \sigma(x,v)=s\}$ then $\mathcal A_s \cup \mathcal A_{-s}$ is fixed by the Zig-Zag walk (see Figure \ref{Fig-ZZd-1}). Therefore, in dimension larger than 1, the Zig-Zag walk is not irreducible on $\Z^d\times \{-1,1\}^d$.

\begin{figure}
\begin{center}
\includegraphics[scale=0.5]{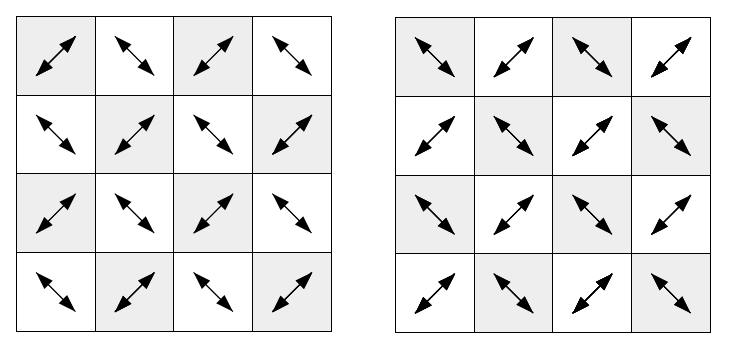}
\caption{Representation of the two irreducible classes of the Zig-ZagZ walk in dimension 2. For instance, if we start in a grey box with velocity $(1,1)$, then whenever we will be in a grey box the velocity will be either $(1,1)$ or $(-1,-1)$ and whenever we will be in a white box the velocity will be $(1,-1)$ or $(-1,1)$. This corresponds to the array on the left.}\label{Fig-ZZd-1}
\end{center}
\end{figure}

\begin{prop}\label{Prop-ZZZd-Irreductible}
Suppose that there exist $R>0$ such that $U(x+v_ie_i)  > U(x)$ for all $i\in\cco 1,d\ccf$, $x\in\Z^d$, $v\in\{-1,1\}^d$ with $x_iv_i > R $.
 Then for all $s\in\{-1,1\}^d$ the Zig-Zag walk on $\Z^d$ associated to $U$ is irreducible on $\mathcal A_s \cup \mathcal A_{-s}$.
\end{prop}

\begin{proof}
Let $s\in\{-1,1\}^d$ be fixed, and let $(x,v)\in\mathcal A_s\cup \mathcal A_{-s}$. Remark that $(x,w)\in\mathcal A_s\cup\mathcal A_{-s}$ if and only if $w\in\{-v,v\}$. We say that we can reach $(y,w)$ from $(x,v)$ if there is a path from $(x,v)$ to $(y,w)$ that has a non-negative probability for the Zig-Zag walk. Starting from  $(x,v)\in\Z^d\times \{-1,1\}^d$, we can reach all the points $(x+nv,v)$ with $n\in \N$. For $n$ large enough, $v_i(x_i+nv_i)>R$ for all $i\in\cco 1,d\ccf$ so that each coordinate has a non-negative probability to flip its velocity in the next step. As a consequence, from $(x+nv,v)$ with such a $n$, $(x+nv+(v+w)/2,w)$ can be reached for all $w\in\{-1,1\}^d$. In particular, if $w=-v$, since $(x,-v)$ can be reached from $(x+nv,-v)$, we see by transitivity that $(x,-v)$ can be reached from $(x,v)$. 


Second, let us show that for all $j\in\cco 1,d\ccf$ and all $a\in\{-1,1\}$, $(x+ae_j,-v+2v_je_j)$ can be reached from $(x,v)$. Remark that this will conclude the proof: indeed, repeating this, then for all $x'\in\Z^d$, there will exist $w'\in\{-1,1\}^d$ such that $(x',w')$ (and then $(x',-w')$ by the previous result) can be reached from $(x,v)$. Since $\mathcal A_s\cup \mathcal A_{-s}$ is fixed by the Zig-Zag walk, if $(x',z)\in \mathcal A_s\cup \mathcal A_{-s}$ then necessarily $z\in\{w',-w'\}$, and thus we will have obtained that all points of $\mathcal A_s\cup \mathcal A_{-s}$ will be reachable from $(x,v)$. 

Hence, fix $j\in\cco 1,d\ccf$, $a\in\{-1,1\}$ and set $(x',v')=(x+ae_j,-v+2v_je_j)$. Since $(x,-v)$ can be reached from $(x,v)$ we can suppose that $v_j=-a$. Let $n_1$ and $n_2$ be large enough so that,  for all $i\in\cco 1,d\ccf$,
\[v_i(x_i+n_1v_i)>R\,,\qquad -v_i'(x_i'-n_2v_i')>R\,,\qquad -v_i'(x_i+(n_1+1)v_i-v_je_j -n_2v_i' ) >R\,.\]
Consider the following path: from $(x,v)$, go to $(x+n_1v,v)$, flip the $j^{th}$ velocity, which gives $(x+(n_1+1)v-v_je_j,-v')$, go straight to 
 $(x+(n_1+1)v-v_je_j-n_2v',-v')$, flip all the velocities but the $j^{th}$, which gives $(x+(n_1+1)v-2v_je_j-n_2v',-v)$, go straight to $(x+v-2v_je_j-n_2v',-v)$, flip the $j^{th}$ velocity, which gives $(x-v_je_j-n_2v',v')$ and go straight to $(x',v')$ (see Figure \ref{Fig-ZZd-2}). It is clear that, in view of the conditions on $n_1$ and $n_2$ the two first flips have a non-negative probability. For the third one, remark that the $j^{th}$ coordinate of $x+v-2v_je_j-n_2v'$ is $x_j'-n_2v_j'$ and that $v_j' = v_j$. Hence, the condition that  $-v_i'(x_i'-n_2v_i')>R$ ensures that the third flip, hence the whole path, has a non-negative probability, which concludes.


\begin{figure}
\begin{center}
\includegraphics[scale=0.5]{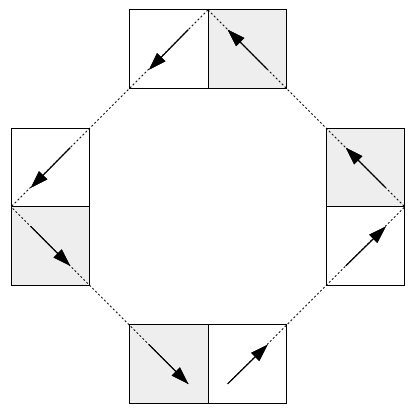}
\caption{An admissible path from $(x,(1,1))$ to $(x-e_1,(1,-1))$.}\label{Fig-ZZd-2}
\end{center}
\end{figure}

%
\end{proof}

\subsubsection{A Lyapunov function}\label{SubSection-ZZZd-Lyapunov}

For some fixed $a,b>0$ and for $x\in\Z^d$, $v\in\{-1,1\}^d$ and $i\in\cco 1,d\ccf$, denote
\[\mathcal V_i(x,v) \ = \ e^{a |x_i| + b\1_{\{x_iv_i>0\}}}\,,\]
and $\mathcal V(x,v) = \sum_{i=1}^d \mathcal V_i(x,v)$. Consider $(X_n,V_n)_{n\in\N}$ the Zig-Zag walk on $\Z^d$ associated to $U:\mathbb Z^d \rightarrow \R$ and denote $Q$ its transition operator, namely
\[Qf(x,v) \ = \ \sum_{w\in\{\pm1\}^d} f\po x + \frac{v+w}{2}, w\pf p(x,v;w)\,.\]

\begin{prop}\label{Prop-ZZZd-Lyapunov}
 Suppose that there exist $R,h>0$ such that for all $i\in\cco 1,d\ccf$, $x\in\Z^d$ and $v\in\{-1,1\}^d$,
 \begin{eqnarray}\label{Eq-ZZZd-Cond-U}
 \po U(x+v_ie_i) - U(x)\pf_+ & \geqslant & h \1_{\{v_i x_i > R\}}\,.
 \end{eqnarray}
Then, for all choice of $a,b>0$ and for all $x\in\Z^d$ and $v\in\{-1,1\}^d$,
\begin{eqnarray}\label{Eq-ZZZd-Lyapunov}
Q \mathcal V(x,v) & \leqslant &\max\po e^{-h+a} +(1-e^{-h})e^{-b},e^{-a} \pf \mathcal V(x,v) + d e^{a(R+1)+b}\,.
\end{eqnarray}
\end{prop}
\begin{proof}
For all  $x\in\Z^d$, $v\in\{-1,1\}^d$,
\[ Q\mathcal V_1(x,v) \ = \  q_1(x,v_1) e^{a|x_1+v_1|+b\1_{\{(x_1+v_1)v_1>0\}}} + \po 1- q_1(x,v_1)\pf e^{a|x_1|+b\1_{\{x_1v_1<0\}}}\,.\]
If $x_1v_1>R$, $q_1(x,v_1) \leqslant e^{-h}$, so that
\begin{eqnarray*}
 Q\mathcal V_1(x,v) &=& q_1(x,v_1) e^{a|x_1|+a+b} + \po 1- q_1(x,v_1)\pf e^{a|x_1|}\\
 & \leqslant & e^{a|x_1|} + e^{-h} \po e^{a|x_1|+a+b} - e^{a|x_1|}\pf\\
 & = & \po e^{-b} + e^{-h+a} - e^{-h-b}\pf \mathcal V_1(x,v)\,.
\end{eqnarray*}
If $x_1v_1<-R$ then $(x_1+v_1)v_1 \leqslant 0$ and $q_1(x,v_1)=1$, and thus
\[  Q\mathcal V_1(x,v) \ = \   e^{a|x_1|-a} \ = \ e^{-a} \mathcal V_1(x,v)\,.\]
If $|x_1v_1|\leqslant R$ then
\[Q\mathcal V_1(x,v) \ \leqslant \ e^{a(R+1)+b}\,.\]
Since $Q$ is the result of $d$ consecutive and identical one-dimensional transitions, we get the result by summing over $i\in\cco 1,d\ccf$.
\end{proof}

Taking $a=h/2$ and $b$ arbitrarily large we get that $Q\mathcal V \leqslant \gamma \mathcal V + C$ with $\gamma<1$ (in fact $\gamma$ arbitrarily close to $e^{-h/2}$), which means that $\mathcal V$ is a Lyapunov function for $Q$.

Remark that similar computations in the case of the continuous-time Zig-Zag process on $\R^d$ show that if $(v_i\partial_{x_i} U(x))_+\geqslant h\1_{\{x_iv_i>0\}}$ for all $x\in\R^d$ such that $|x_i|>R$ then
\[\widetilde{\mathcal V}(x,v) \ := \ \sum_{i=1}^d e^{a|x_i| + b \varphi(v_ix_i)}\,,\]
where $\varphi(s)$ is some smooth approximation of $sign(s)$, is a Lyapunov function for the continuous-time process. Note that this condition on $U$ is not covered by \cite[Condition 3]{BierkensRobertsZitt} since the latter constrains $|\na U(x)|$ to go to infinity at infinity, excluding Laplace-tail distributions. Hence, our computations extends the scope of Lemma 2 (hence Theorem 2) of \cite{BierkensRobertsZitt}. Note that condition \eqref{Eq-ZZZd-Cond-U} holds when $U(x) = \sum_{i=1}^d |x_i|$ but not when $U(x) = |x|$. Of course, condition \eqref{Eq-ZZZd-Cond-U} roughly means that the different coordinates are more or less independent at infinity and thus it is not surprising that we recover, in a discrete-space case, the one-dimensional computations of \cite[Proposition 2.8]{FontbonaGuerinMalrieu}.

In the following, under condition \eqref{Eq-ZZZd-Cond-U}   we fix $a=h/2$ and $b$ such that $e^{-b} = e^{-h/4}-e^{-h/2}$, in which case \eqref{Eq-ZZZd-Lyapunov} implies
\[Q\mathcal V (x,v) \ \leqslant \ e^{-h/4} \mathcal V(x,v) + \po 1-e^{-h/4}\pf d \frac{e^{h(R/2+1)}}{\po 1-e^{-h/4}\pf^2}\,.\]
This classically yields uniform in time exponential moment bounds on the Zig-Zag walk, since for all $n\in\N$,
\[Q^n \mathcal V(x,v) \ \leqslant e^{-n h/4} \mathcal V(x,v) + \po 1-e^{-nh/4}\pf d \frac{e^{h(R/2+1)}}{\po 1-e^{-h/4}\pf^2}\,\]
and thus, if $(X_0,V_0)=(x,v)$, we bound
\[\mathbb E \po e^{\frac{h}{2d} |X_n|} \pf \leqslant \mathbb E \po e^{\frac{h}{2d} \sum_{i=1}^d |X_{n,i}|} \pf \leqslant \frac1d  \mathbb E \po \mathcal V(X_n,V_n)\pf  =  \frac1d Q^n\mathcal V(x,v)\,.\]

\subsubsection{Long-time convergence}\label{Section:longtimeZZd}

For $x\in\Z^d$ and $s\in\{-1,1\}^d$, let $v(x,s)$ be the (unique) vector of $\{-1,1\}^d$ such that $\sigma(x,v(x,s))=s$. Then $\mathcal A_s   = \{(x,v(x,s)) \,, x\in\Z^d\}$ and $ \mathcal A_{-s} = \{(x,-v(x,s))\,, x\in\Z^d\}$. Consider on $\mathcal A_s$ the probability measure
\[\mu_s(x,v) \ =\ \pi(x)\delta_{v(x,s)}(v) \ = \ \frac{\mu(x,v)\1_{(x,v)\in \mathcal A_s}}{\mu(\mathcal A_s)}\,.\]
For a given $\mathcal W:\Z^d\times\{-1,1\}^d \rightarrow [1,+\infty)$, we endow $\mathcal P_{\mathcal W}:=\{\mu\in\mathcal P(\Z\times \{-1,1\}^d)\ : \mu(\mathcal W) < +\infty\}$ with the norm
\[\| \mu - \nu\|_{\mathcal W} \ := \ \sup_{|f|\leqslant \mathcal W}\po \mu(f)-\nu(f)\pf\,,\]
which makes it complete.

\begin{thm}\label{TheoremZZZdErgodic}
 Suppose that $U$ admits a strict local minimum at some $x_*\in\mathbb Z^d$ and that there exist $R,h>0$ such that \eqref{Eq-ZZZd-Cond-U} holds for all $i\in\cco 1,d\ccf$, $x\in\Z^d$ and $v\in\{-1,1\}^d$. Set $\mathcal W(x,v) = \sum_{i=1}^d e^{h|x_i|/2}$.  Then, there exist $C>0$ and $\rho\in(0,1)$ such that for all $x\in\Z^d$, $v\in\{-1,1\}^d$ and  $n\in\N_*$,
 \[\| \delta_{(x,v)}Q^{2n} - \mu_{\sigma(x,v)}\|_{\mathcal W} \ \leqslant \ C \rho^{n} \mathcal W(x,v)\,.\]
\end{thm}

\begin{proof}
Let $(x,v),(x',v')\in\mathcal A_s$ for some $s\in\{-1,1\}^d$. Proposition \ref{Prop-ZZZd-Irreductible} gives a path from $(x,v)$ to $(x',v')$ whose transitions are non-negative under $Q$. Since $Q(\mathcal A_s)\subset Q(\mathcal A_{-s})$, the length of such a path is necessarily even, from which $Q^2$ is irreducible on $\mathcal A_s$. The path $(x_*,v(x_*,s))\rightarrow(x_*,-v(x_*,s))\rightarrow(x_*,v(x_*,s))$ having a non-negative probability under $Q$, $Q^2$ is aperiodic on $\mathcal A_s$. As a consequence, for all $(x,v),(x',v')\in\mathcal A_s$, there exist $n_0$ such that $Q^{2n}\po (x,v),(x',v')\pf >0$ for all $n\geqslant n_0$. From Proposition \ref{Prop-ZZZd-invariance}, $\mu_s$ is invariant for $Q^2$.

Consider $\mathcal V$ as defined in Section \ref{SubSection-ZZZd-Lyapunov} with $a=h/2$ and $b$ large enough, so that
 \begin{eqnarray}\label{Eq-DemoZZdErgoLyap}
 Q^{2n}\mathcal V(x,v) & \leqslant & \gamma^{2n} \mathcal V(x,v) + \po 1 - \gamma^{2n}\pf C 
 \end{eqnarray}
for some $\gamma\in(0,1)$, $C>0$. The set $\{\mathcal V \leqslant 4C\}$ is finite. Fix any point $x'\in\Z^d$ (say, $x'=x_*$) and let
\[m\ :=\ \max_{\mathcal V (x,v) \leqslant 4C} \min\{k\in\N\ : \ Q^{2r}\po (x,v),(x',v(x',\sigma(x,v))\pf>0\ \forall r\geqslant k\}\,.\]
Then 
 \begin{eqnarray}\label{Eq-DemoZZdErgoDoeblin}
\min\{Q^{2m}\po (x,v),(x',s)\pf & : &  (x,v)\in \mathcal A_s,\ \mathcal V (x,v) \leqslant 4C\} \ > \ 0\,.
 \end{eqnarray}
From \cite[Theorem 1.2]{HairerMattingly2008} applied to $Q^{2m}$, the Foster-Lyapunov condition \eqref{Eq-DemoZZdErgoLyap} (applied with $n=m$) and the Doeblin condition \eqref{Eq-DemoZZdErgoDoeblin} imply the existence of $\rho\in(0,1)$  and $C'>0$ such that for all $x\in\Z^d$, $v\in\{-1,1\}^d$ and  $k\in\N_*$,
 \[\| \delta_{(x,v)}Q^{2km} - \mu_{\sigma(x,v)}\|_{\mathcal V } \ \leqslant \ C' \rho^k \mathcal V (x,v)\,.\]
 In fact, even without the Doeblin condition, following the proof of \cite[Theorem 1.2]{HairerMattingly2008}, we also get that the Lyapunov condition given by Proposition \ref{Prop-ZZZd-Lyapunov}  alone implies the following: there exist $\tilde C>0$ such that, for all probability measures $\nu,\nu'\in\mathcal P_{\mathcal V}$,
 \begin{eqnarray}\label{Eq-ContractionQZZd}
 \| \nu Q - \nu' Q\|_{\mathcal V} &  \leqslant &  \tilde C \|\nu  - \mu_s\|_{\mathcal V }\,,
 \end{eqnarray}
 and in particular 
 \[\| \nu Q^2 - \mu_s\|_{\mathcal V } \ = \ \| \nu Q^2 - \mu_s Q^2\|_{\mathcal V }    \  \leqslant \   \tilde C^2 \|\nu  - \mu_s\|_{\mathcal V } \,. \]
 Then for all $n\in\N$, considering the Euclidian division $n=km+r$ we get that
  \[\| \delta_{(x,v)}Q^{2n} - \mu_{\sigma(x,v)}\|_{\mathcal V }  \ = \ \| \delta_{(x,v)}Q^{2r}Q^{2km} - \mu_{\sigma(x,v)}\|_{\mathcal V }  \ \leqslant \ \tilde C^{2m}C' \rho^{n/m-1} \mathcal V (x,v)\,.\]
  The equivalence betwenn $\mathcal V$ and $\mathcal W$, hence between $\|\cdot\|_{\mathcal V}$ and $\|\cdot\|_{\mathcal W}$, concludes.
\end{proof}

\subsubsection{Asymptotic theorems}\label{SubSectionTCLZZZd}

\begin{thm}
 Suppose that $U$ admits a strict local minimum and that there exist $R,h>0$ such that \eqref{Eq-ZZZd-Cond-U} holds for all $i\in\cco 1,d\ccf$, $x\in\Z^d$ and $v\in\{-1,1\}^d$. Let $f:\Z^d \rightarrow \R$ be such that $\|f/\mathcal W\|_\infty<\infty$, where $\mathcal W$ is defined in Theorem \ref{TheoremZZZdErgodic} (here and below we identify $f$ with the function $(x,v)\in \Z^d\times\{-1,1\}^d\mapsto f(x)$). Consider the Zig-Zag walk on $\Z^d\times\{-1,1\}^d$ associated to $U$ with some initial condition $(x,v)$. Then, almost surely,
 \[\frac1n\sum_{k=0}^n f(X_k)  \ \underset{n\rightarrow \infty}\longrightarrow \ \pi(f)\,.\]
 If, moreover, $\|f/\sqrt{\mathcal W}\|_\infty<\infty$, then
\[\po \frac1{\sqrt n}\po \sum_{k=0}^{\lfloor n t\rfloor} f(X_k) - \pi(f)\pf \pf_{t\geqslant 0}\  \overset{law} {\underset{n\rightarrow +\infty}\longrightarrow}\ (\sigma_f B_t)_{t\geqslant0}\]
for some $\sigma_f\geqslant 0$,  where $(B_t)_{t\geqslant0}$ is a one-dimensional Brownian motion.
\end{thm}
\begin{proof}
For the first part of the Theorem, simply decompose 
\[\frac1n\sum_{k=0}^n f(X_k)  \ = \ \frac1n\sum_{k=1}^{\lfloor n/2\rfloor} f(X_{2k})  +  \frac1n\sum_{k=0}^{\lfloor (n-1)/2\rfloor}  f(X_{2k+1}) \]
From Proposition \ref{TheoremZZZdErgodic} and the law of large numbers for $\mathcal W$-regular ergodic Markov chains, these terms almost surely converge respectively to $\mu_{\sigma(x,v)}(f)/2$ and $\mu_{-\sigma(x,v)}(f)/2$ (since $(X_1,V_1)\in\mathcal A_{-\sigma(x,v)}$ almost surely), which are both equal to $\pi(f)/2$.

For the second part, note that by the Jensen inequality and Proposition \ref{Prop-ZZZd-Lyapunov},
\[Q\sqrt{\mathcal V} \ \leqslant \ \sqrt{Q\mathcal V} \ \leqslant \ \sqrt{\gamma \mathcal V + C} \leqslant \sqrt\gamma \sqrt \mathcal V + \sqrt C\]
for some $\gamma\in(0,1)$ and $C>0$. Hence, $\sqrt{\mathcal V}$ is still a Lyapunov function for $Q$ and the results established with $\mathcal V$ and $\mathcal W$ in the previous section also hold with $\sqrt \mathcal V$ and $\sqrt{\mathcal W}$. Let $f:\Z^d\times\{-1,1\}^d\rightarrow \R$ with $\|f/\sqrt{\mathcal V}\|_\infty <\infty$. From Theorem \ref{TheoremZZZdErgodic} (applied with $\sqrt{\mathcal W}$), for all $(x,v)\in\Z^d\times\{-1,1\}^d$ and all $n\in \N$,
\[|Q^{2n} f(x,v) - \mu_{\sigma(x,v)}(f)| \ \leqslant \ C\rho^n \|f/\sqrt{\mathcal W}\|_\infty \sqrt{\mathcal W(x)}\]
and, using \eqref{Eq-ContractionQZZd},
\[|Q^{2n+1} f(x,v) - \mu_{-\sigma(x,v)}(f)| = |Q^{2n} Qf(x,v) - \mu_{\sigma(x,v)}Q(f)| \ \leqslant \ \tilde C C\rho^n \|f/\sqrt{\mathcal W}\|_\infty \sqrt{\mathcal W(x)}\,.\]
Here we used that $\mu_{s}Q = \mu_{-s}$, which can be obtained from
\[\mu_{s}Q = \mu_{s} Q^{2n} Q = \mu_{s}Q Q^{2n} \underset{n\rightarrow\infty}\longrightarrow \mu_{-s}\,,\]
where the limit holds in $\mathcal P_{\mathcal W}$ thanks to Proposition \ref{TheoremZZZdErgodic} together with the fact that the support of $\mu_{s}Q $ is included in $\mathcal A_{-s}$. We have thus obtained that 
\[g(x,v) \ := \ \sum_{n\in\N} \po Q^n f(x,v) - \mu_{(-1)^n\sigma(x,v)}(f)\pf \]
is well-defined and satisfies $\|g/\sqrt{\mathcal W}\|_\infty  \leqslant C \| f/\sqrt{\mathcal W}\|_\infty $ for some $C>0$ independent from $f$. Now suppose that in fact $f$ is a function of space alone, i.e. $f(x,v)=f(x)$. In that case $\mu_s(f) = \mu(f)$ for all $s\in\{-1,1\}^d$ and thus
\[Qg(x,v) \ = \ \sum_{n\in\N} \po Q^{n+1} f(x,v) - \mu(f)\pf \ = \ g(v,x) - f(x) + \mu(f)\,,\]
in other words $g$ is the Poisson solution associated to $Q$ and $f$. Since $\sqrt{\mathcal W}\in L^2(\mu)$, so does  $g$, and \cite[Theorem 3.1]{Maigret} 
 concludes.
 \end{proof}

\section{Numerical scheme for hybrid kinetic samplers}\label{Section-Schema-Euler}

\subsection{The continuous-time processes}\label{Section-COntinuHybride}

In this section we consider a class of kinetic processes for MCMC that can have a jump, drift and/or diffusion
component at the same time in their dynamics.
 They are defiend as follows.

Let $U\in \mathcal C^\infty(\T^d)$ (where $\T=\R/\Z$ is the $1$-periodic torus) and denote by $\mu$ the Gibbs measure on $E=\T^d\times\R^d$ associated to the Hamiltonian $H(x,v) = U(x) + |v|^2/2$, namely the probability law on $E$ with density proportional to $\exp(-H)$. Suppose that $\na U(x) = \sum_{i=0}^N F_i(x)$ where $N\in\N$ and $F_i\in\mathcal C^\infty(\T^d, \R^d)$ for  all $i\in\cco 0,N\ccf$. Consider the operator $L$ defined for all $f\in\mathcal C^2(E)$ by
\begin{eqnarray}\label{Eq-Mixed-kinetic-L}
L &=&  A_1+A_2 + \sum_{i=1}^N A_{3,i} + \gamma A_4 + \lambda A_5 \,,
\end{eqnarray}
where
\begin{eqnarray*}
A_1 f(x,v) & = &  v\cdot \na_x f(x,v)  \\
A_2 f(x,v) & = & -  F_0(x)   \cdot \na_v f(x,v)\\
A_{3,i} f(x,v) & = & \po v\cdot F_i(x)\pf_+ \po f\po x,R_i(x,v)\pf - f(x,v)\pf\qquad \forall i\in\cco 1,N\ccf \\
A_4 f(x,v) & = & v \cdot \na_v f(x,v) + \Delta_v f(x,v)\\
A_5 f(x,v) & = & \int_{\R^d} \po f(x,w) -f(x,v)\pf \nu_d(\dd w)
\end{eqnarray*}
and where $\gamma,\lambda\geqslant 0$, $\nu_d$ is the standard $d$-dimensional Gaussian distribution and
\[R_i(x,v) \ = \ v - 2\po \frac{F_i(x)\cdot v}{|F_i(x)|^2}F_i(x)\pf \1_{F_i(x)\neq 0}\]
is the orthogonal reflection of $v$ with respect to $F_i(x)$. We call respectively $A_1$ the transport operator, $A_2$ the drift one, $A_{3,i}$ the $i^{th}$ bounce one, $A_4$ the Ornstein-Uhlenbeck (or friction/dissipation) one and $A_5$ the refreshment one.

As particular cases, many usual kinetic processes used in MCMC algorithms can be recovered:

\begin{itemize}
\item $F_0=\na_x U$ and $\lambda=\gamma=0$ corresponds to the Hamiltonian dynamics.
\item $F_0=\na_x U$, $\gamma>0$ and $\lambda = 0$  to the Langevin diffusion.
\item $F_0=\na_x U$, $\gamma=0$ and $\lambda >0 $ to the Hybrid Monte Carlo (HMC) algorithm.
\item $F_0=0$, $N=1$, $F_1=\na_x U$, $\gamma = 0$ and $\lambda>0$  to the Bouncy Particle sampler.
\item $F_0=0$, $N=d$, $F_i=\na_{x_i}U e_i$ (recall $e_i$ denotes the $i^{th}$ vector of the canonical basis of $\R^d$), $\gamma=\lambda= 0$  to the Zig-Zag process.
\end{itemize}

We could also consider other kinds of jump mechanisms, like the randomized bounces of \cite{RobertWu,Michelforward}, or different kinds of relaxation operators in the velocity operator rather than $A_4$ and $A_5$, for instance refreshment of velocities coordinate by coordinate, or partial refreshments for which, at exponential random times with parameter $\lambda$, the velocities $v$ jumps to $(1-\alpha) v + \sqrt{\alpha} G$ where $G\sim\nu_d$ (varying $\alpha\in(0,1]$ and $\lambda$ interpolates between $A_5$ and $A_4$, the latter being the limit $\alpha\rightarrow 0$ and $\lambda\rightarrow +\infty$ with $\lambda \alpha = 1$). However, this would just make the notations heavier and the presentation more confused, without adding any particularly new idea with respect to the discussion to come, and thus we stick to \eqref{Eq-Mixed-kinetic-L}.

For MCMC purposes, the law of the process associated to a generator $L$ of the form above should converge in large times toward the target measure $\mu$. This is established in the next two results. However, this is not the main motivation of this section, which is the study of  the numerical sampling of such a process, and thus we only give sketches of proof and references for these results.

\begin{prop}
The set of compactly supported smooth functions $\mathcal C^\infty_c(E)$ is a core for $L$ and $\mu$ is invariant for $L$.
\end{prop}
\begin{proof}
The proof that  $\mathcal C^\infty_c(E)$ is a core for $L$ is similar to the case of the Bouncy Particle Sampler in \cite[Theorem 21 and Proposition 23]{DurmusGuillinMonmarcheToolbox}. From that, the invariance of $\mu$ straightforwardly follows from the fact that $\mu( L f) =0$ for all compactly supported $\mathcal C^2$ functions, which is easily checked through integration by part (see also \cite[Section 1.4]{MonmarcheRTP}).
\end{proof}

Similarly to Section \ref{Section:longtimeZZd}, we denote $\mathcal W(x,v) = 1+|v|^2$, $\mathcal P_{\mathcal W} = \{\nu\in\mathcal P(E),\ \nu(\mathcal W) <\infty\}$, $\| \nu_1-\nu_2\|_{\mathcal W} = \sup_{|f|\leqslant \mathcal W}|\nu_1 f - \nu_2 f|$.  

\begin{prop}\label{Prop:HybridCVtempslong}
Suppose that either $\gamma> 0$ or $\lambda> 0$. Then there exist $\kappa,C>0$ such that the semi-group $(P_t)_{t\geqslant 0}$ associated to $L$ satisfies, for all $t\geqslant 0$ and $\nu \in \mathcal P_{\mathcal W}$,
\[\|\nu P_t - \mu \|_{\mathcal W}  \ \leqslant \ C e^{-\kappa t} \|\nu-\mu\|_{\mathcal W}\]
\end{prop}

\begin{proof}
Similarly to Theorem \ref{TheoremZZZdErgodic}, this is classically obtained by combining a Doeblin and a Foster-Lyapunov conditions \cite{HairerMattingly2008}. 

When $\gamma> 0$ or $\lambda> 0$, the Doeblin condition is obtained through controllability arguments (a notable fact is that under quite general conditions on $U$ the Zig-Zag process  is irreducible even if $\lambda = \gamma =0$, see  \cite{BierkensRobertsZitt}). Remark that the bounce operators play no role here since if, uniformly in $(x,v)$ in a compact of $E$, $\delta_{(x,v)} e^{t_0 (A_1+A_2 +  \gamma A_4 + \lambda A_5)}$ is bounded below at some time $t_0>0$ by $c_0$ times the uniform measure on some compact set with some $c_0>0$, then $\delta_{(x,v)}e^{t_0 L}$ is bounded below by $c_0 e^{-t_0\sum_{i=1}^N\| F_i\|_{\infty}}$ times this uniform measure, where we have used that the total bounce rate is bounded above by $\sum_{i=1}^N\| F_i\|_{\infty}$ so that there is a probability at least $e^{-t_0\sum_{i=1}^N\| F_i\|_{\infty}}$ that no bounce occurs on the time interval $[0,t_0]$.

The Lyapunov condition stems from the fact that for all $(x,v)\in E$,
\begin{eqnarray*}
A_1 \mathcal W(x,v) & = & 0\\
A_2 \mathcal W(x,v) & \leqslant & 2\|F_0\|_\infty   |v| \\
\qquad A_{3,i} \mathcal W(x,v) & = & 0\qquad \forall i\in\cco 1,N\ccf\\
A_4 \mathcal W(x,v) & = &  -2|v|^2 + 2d\\
A_5 \mathcal W(x,v) & = &  -|v|^2 + d \,,
\end{eqnarray*}
and thus $L\mathcal W \leqslant -\min(2\gamma,\lambda)/2 \mathcal W + C$ for some $C>0$.

\end{proof}

As announced, from now and in the rest of Section \ref{Section-Schema-Euler} we  focus on the question of the  the numerical sampling of the process with generator $L$.

When $F_0=0$, and $\gamma=0$, the process is a piecewise deterministic velocity jump process. Between two random jumps, the process simply follows the flow $(x,v)\mapsto (x+tv,v)$, so that the jump time $T_i$ associated to the vector field $F_i$ follows the law
\[\mathbb P\po T_i > t\pf = \exp\po - \int_0^t \po v\cdot F_i(x+sv)\pf_+\dd s\pf\,.\]
Provided that for all $x,v\in\R^d$ and $s>0$, $\po v\cdot F_i(x+sv)\pf_+\leqslant \varphi_{x,v}(s)$ for some function $\varphi_{x,v}$ such that $\int_0^t \varphi_{x,v}(s)\dd s$ can be computed, the continuous-time thinning algorithm  allows for an exact simulation of the jump times, hence of the process \cite{Thieullen2016}. The absence of discretization bias on the invariant measure is an argument in favour of these kinetic processes. This is no longer the case when $F_0\neq 0$, except in very particular cases (e.g. harmonic oscillators) where the ODE $\partial_t( x , v)=(v,-F_0(x))$ can be explicitly solved. In most cases, this ODE is solved numerically and thus the simulation for the stochastic process is not exact. In order to conserve the invariant measure and suppress the bias, a Metropolis step can be added \cite{BouRabee}, but this slows down the motion of the process, increases the variance and may thus be counter-productive, as observed when comparing Metropolis Adjusted Langevin algorithm and Unadjusted Langevin Algorithm \cite{DurmusMoulines}. 

So why mix a deterministic drift and jump mechanisms if this prohibits exact simulation? The motivation is given by the possible numerical gain given by thinning, as presented in Section \ref{Section-thinning}. Indeed, as said before, exact simulation requires a bound on the jump rate, and a poor (i.e. large) bound leads to many jump proposals per time unit and a very low efficiency. In particular, jump mechanisms are not adapted for fastly-varying potentials, like potentials used in molecular dynamics with a singularity at zero (Lennard-Jones, Coulomb\dots). On the other hand, jump mechanisms are very efficient with Lipschitz potentials. So, a mixed drift/jump part is interesting as soon as the potential exhibits different scales, like fast-varying but numerically cheap parts together with Lipschitz but numerically intensive parts. This is similar to the idea of multi-time-steps algorithms \cite{TuckermanRossiBerne,GibsonCarter} with somehow random adaptive time-steps, except that now we are simply going to discretize (with a unique time-step, no subtlety here) a continuous-time process which is ergodic with respect to the target law and thus there shouldn't be any resonance problem as exhibited by multi-time-steps algorithms. Besides, for the applications in molecular dynamics that have motivated this question \cite{MonmarcheBouncyChimie}, due to stability issues raised by very fast oscillations in some parts of the system, the time-step is anyway constrained to be very small, as compared to a high variance that comes from the problem of exploring a complex, high-dimensional, multi-scale, metastable landscape. So, exact simulation is not necessarily our objective.

\subsection{A Strang splitting scheme}\label{SubSectionStrang}

A Strang splitting scheme to compute the evolution given by a generator $L=L_1+L_2$ is based on the fact that, formally,
\[e^{t(L_1+L_2)} \  =\ e^{tL_1/2}e^{tL_2}e^{tL_1/2} + \underset{t\rightarrow 0}o(t^2)\,.\]
Hence, if we can simulate exactly a process with generator $L_1$ and $L_2$, or more generally if we have second-order approximations of those, we get a second-order scheme for $L_1+L_2$. Using twice this fact,
\[e^{t(L_1+L_2+L_3)} \  =\ e^{tL_1/2}e^{t(L_2+L_3)}e^{tL_1/2} + \underset{t\rightarrow 0}o(t^2) \  =\ e^{tL_1/2}e^{tL_2/2}e^{tL_3} e^{tL_2/2}e^{tL_1/2} + \underset{t\rightarrow 0}o(t^2) \,.\]
In particular, from the considerations developed in Section \ref{SubsectionInvariantMeasure}  (and Theorem \ref{Thm-Kallenberg}), the equilibrium of the corresponding Markov chain is close (in some senses) to the Gibbs measure $\mu$ at order $\delta^2$, where $\delta$ is the timestep.

For instance, for the Langevin diffusion, which corresponds to $L$ given by \eqref{Eq-Mixed-kinetic-L} with $\gamma >0$, $\lambda= 0$, $N=0$ and $F_0=\na_x U$, many splitting schemes can be considered, see e.g. the discussions in \cite{Leimkuhler,BouRabee}. A very precise study, both theoretical and empirical, of all the possible schemes for mixed jump/diffusion processes is beyond the scope of the present paper, and we will only consider one particular choice. Consider the splitting
\[L_1 = A_1\,,\qquad L_2 = A_2+\gamma A_4\,,\qquad L_3 = \sum_{i=1}^N A_{3,i} + \lambda A_5\,.  \]
Each of the three evolutions corresponding to $e^{tL_i}$, $i=1,2,3$ can be sampled exactly (remark that, in particular cases, the velocity jump process corresponding to $e^{t(L_1+L_3)}$ could also be sampled exactly). As a consequence, for a given time-step $\delta>0$, we consider the kinetic walk $(X_n,V_n)_{n\in\N}$ whose transition is  defined as follows:
\begin{enumerate}
\item Set $\widetilde X_n = X_n + \delta V_n/2$.
\item If $\gamma> 0$, set 
\[\widetilde  V_n \ =\  e^{-\gamma \delta/2 } V_n - \po 1 - e^{-\gamma \delta /2 } \pf F_0(\widetilde X_n) + \sqrt{\po 1 - e^{-\gamma \delta/2  } \pf  } G\]
 with $G$ a standard Gaussian random variable. If $\gamma=0$, set $\widetilde  V_n = V_n - \delta F_0(\widetilde X_n)/2  $.
\item Set $\widehat  V_{n} = W_\delta$ where $(Y_t,W_t)_{t\in[0,\delta]}$ is a Markov chain with generator $L_3$ and initial condition $(Y_0,W_0) = (\widetilde X_n,\widetilde  V_n)$ (so that $Y_t = \widetilde X_n$ for all $t\in[0,\delta]$). 
\item If $\gamma> 0$, set 
\[ V_{n+1} \ =\  e^{-\gamma \delta/2  } \widehat V_n - \po 1 - e^{-\gamma \delta/2  } \pf F_0(\widetilde X_n) + \sqrt{\po 1 - e^{-\gamma \delta/2  } \pf  } G'\]
 with $G'$ a standard Gaussian random variable. If $\gamma=0$, set $V_{n+1} = \widehat  V_n - \delta F_0(\widetilde X_n)/2  $.
\item Set $X_{n+1} = \tilde X_n + \delta V_{n+1}/2 = X_n + \delta (V_n+V_{n+1})/2$.
\end{enumerate}
(Of course, implicitly, the variables $G$, $G'$ and $(W_t)_{t\in[0,\delta]}$ are  independent one from the other and from the past trajectory). Then, denoting $Q_{\delta}$ the transition operator of $(X_n,V_n)$ defined by $Q_{\delta} f(x,v) = \mathbb E\po f(X_1,V_1)\ | \ (X_0,V_0)=(x,v)\pf$ for all measurable bounded function $f$, one has
\begin{eqnarray*}
Q_{\delta} & = &  e^{\delta L_1/2}e^{\delta L_2/2}e^{\delta L_3} e^{\delta L_2/2}e^{\delta L_1/2} \,.
\end{eqnarray*}
To make the link with the discussion of Section \ref{Section-general}, for all measurable bounded  $f$,
\[Q_\delta f(x,v) \ = \ \int f\po x+ \frac{v+w}{2},w\pf p_{\delta}(x,v;\dd w)\,, \] 
where $p_{\delta}$ is defined by the fact that, in the algorithm above, for all $n\in\N$ the law of $V_{n+1}$ is $p_{\delta}(X_n,V_n;\cdot)$.

\begin{prop}\label{Prop:hybrid}
Let $(Y_t,W_t)_{t\geqslant 0}$ be a Markov process associated to $L$ and, for all $\delta>0$, let $(X_n^\delta,V_n^\delta)_{n\geqslant 0}$ be a Markov chain with transition operator $Q_\delta$ and initial condition $(X_0^\delta,V_0^\delta)=(Y_0,W_0)$. Then
\[\po  X^\delta_{\lfloor t/\delta\rfloor},  V^\delta_{\lfloor t/\delta\rfloor}\pf_{t\geqslant 0} \ \overset{law}{\underset{\delta\rightarrow 0}\longrightarrow}\  (Y_t,W_t)_{t\geqslant 0}\,.\]
\end{prop}
\begin{proof}
We start by a smoothing/truncation step similar to the study in \cite[Theorem 21]{DurmusGuillinMonmarcheToolbox}. For all $\varepsilon>0$ small enough and all $i\in\cco 1,N\ccf$, we consider the jump rate on $E$ given by
\[\lambda_{i,\varepsilon}(x,v) \ = \ \frac{\po v\cdot F_i(x) - \varepsilon\pf_+^2 }{\varepsilon + \po v\cdot F_i(x) - \varepsilon\pf_+}\]
and the associated regularized bounce operator 
\[A_{3,i,\varepsilon} f(x,v)  \ =\   \lambda_{i,\varepsilon}(x,v) \po f\po x,R_i(x,v)\pf - f(x,v)\pf\,. \]
We also consider $\nu_{d,\varepsilon}$ the standard Gaussian law conditioned on $\{|v|\leqslant 1/\varepsilon\}$ and the corresponding truncated refreshment operator 
\[A_{5,\varepsilon} f(x,v) \ = \ \int_{\R^d} \po f(x,w) -f(x,v)\pf \nu_{d,\varepsilon}(\dd w)\,.\]
Then, $L_{\varepsilon} :=  A_1+A_2 +   \sum_{i=1}^N A_{3,i,\varepsilon}  +   \gamma A_4 + \lambda A_{5,\varepsilon}$ is   such that
\[\sup_{\|f\|_\infty \leqslant 1} \| Lf - L_{\varepsilon} f\|_\infty \leqslant C \varepsilon\]
for some $C>0$. The two processes $(Y_t,W_t)_{t\geqslant 0}$ and $(Y_t^\varepsilon,W_t^\varepsilon)_{t\geqslant 0}$ with respective generators $L$ and $L_\varepsilon$ can be defined simultaneously following the synchronous coupling detailed in \cite[Section 6]{DurmusGuillinMonmarcheToolbox}, in such a way that $(Y_t,W_t)=(Y_t^\varepsilon,W_t^\varepsilon)$ up to a random time that is stochastically bounded above by an exponential variable with parameter $C\varepsilon$. In other words, for all $t\geqslant 0$,
\[\mathbb P \po (Y_s,W_s)=(Y_s^\varepsilon,W_s^\varepsilon)\, \forall s\in[0,t]\pf \ \geqslant \ 1 - e^{-C\varepsilon t}\,.\]
In particular $(Y_t^\varepsilon,W_t^\varepsilon)_{t\geqslant 0}\overset{law}{\underset{\varepsilon\rightarrow 0}\longrightarrow}(Y_t,W_t)_{t\geqslant 0}$.
 The same coupling argument holds for a chain $(X_n^{\delta,\varepsilon},Y_n^{\delta,\varepsilon})_{n\in\N}$ with transition operator $Q_{\delta,\varepsilon}$ defined like  $Q_\delta$ but after smoothing/truncation. 
%
 It is thus sufficient to prove Proposition~\ref{Prop:hybrid} in the case where $A_5$ and $(A_{3,i})_{\in\cco 1,N\ccf}$ are replaced by the smooth and truncated operators.

  Denote $\mathcal B = \mathrm{vect}\po\{A_1,A_2,A_4,A_{5,\varepsilon}\}\cup\{A_{3,i,\varepsilon}\ : \ i\in\cco 1,N\ccf\}\pf$. If $f\in\mathcal C_c^\infty(E)$ then $B f\in\mathcal C_c^\infty(E)$ for all $B\in\mathcal B$ (this is false without the smoothing/truncation step, which is the reason why it has been added). Since $\mathcal C_c^\infty(E)$ is in the strong domain of all operators of $\mathcal B$, we can use  for $B_1,B_2\in\mathcal B$ the decomposition
\begin{multline*}
e^{tB_1}e^{tB_2} f - f - t(B_1+B_2) f  = e^{tB_1}\po e^{tB_2}f-f-tB_2 f\pf + e^{tB_1} f - f - tB_1 f\\
 + t \po e^{tB_1}B_2 f - B_2 f\pf
\end{multline*}
and the fact $\|e^{tB} f\|_\infty \leqslant \|f\|_\infty$ for all $B\in\mathcal B$ to get that
\[\| e^{tB_1}e^{tB_2} f - f - t(B_1+B_2) f \|_\infty  \ = \ \underset{t\rightarrow0}o(t^2)\]
for all  $f\in\mathcal C_c^\infty(E)$. Using this repeatedly, we get that
\[\|\frac{1}{\delta} \po Q_{\delta,\varepsilon} f - f\pf - L_\varepsilon f\|_\infty \underset{\delta \rightarrow 0}\longrightarrow 0\]
for all $f\in\mathcal C_c^\infty(E)$, and \cite[Theorem  17.28]{Kallenberg} concludes.
\end{proof}


\subsection{Computational complexity}\label{SubSectionLJ}

Let us informally discuss  the efficiency of the algorithm introduced in the previous section that defines the transition associated with the operator $Q_\delta$. It should be compared with the transition of a numercial scheme for classical processes, like Langevin or Hamiltonian dynamics. The numerical cost mainly comes from the computation of the forces. For the Langevin or Hamiltonian dynamics, at each time step, $\nabla U$ is computed once. Similarly, for a transition associated with $Q_\delta$,  $F_0$ is evaluated once per step, at $x=\widetilde X_n$. What about $(F_i)_{i\in \cco 1,N\ccf}$?

First, consider a naive construction of the third step of the algorithm, namely the construction of $W_\delta$ when $(Y_t,W_t)_{t\geqslant 0}$ is a continuous-time Markov chain with generator $L_3$ and initial condition $(x,w)\in\R^{d}\times\R^d$. For all $t\geqslant 0$, $Y_t=x$. Start by sampling a Poisson process with intensity $\lambda$ on $[0,\delta]$, consider $T_0$ its last jump in $[0,\delta]$ (with $T_0=0$ if there is no jump, in particular if $\lambda=0$). If $T_0=0$, set $W_{T_0}=w$, else draw $W_{T_0}$ according to the standard Gaussian distribution. Suppose by induction that $T_n$ and $W_{T_n}$ have been defined for some $n\in\N$.  Let $(E_{n,i})_{i\in\cco 1,N\ccf}$ be i.i.d. exponential random variable and let 
\[\forall i\in\cco 1,N\ccf\,, \, T_{n,i} = \frac{E_{n,i}}{\po W_{T_n}\cdot F_i(x)\pf_+}\,, \qquad T_{n+1} = \min_{i\in\cco 1,N\ccf} T_{n,i}\,.\]
Let $i_n\in\cco 1,N\ccf$ be such that $T_{n,i_n}=T_{n+1}$ ($i_n$ is almost surely unique), set $W_t = W_{T_n}$ for all $t\in[T_n,T_{n+1})$ and $W_{T_{n+1}} = R(x,W_{T_n})$. Remark that the norm of $W_t$ is conserved at a jump time, so that the jump rate $\sum_{i=1}^N (W_t\cdot F_i(x))_+$ is bounded and $\max\{n\in\N\ : \ T_n<\delta\}$ is almost surely finite, and thus $W_\delta$ is defined after a finite number of jumps.

This is a correct construction, but it relies on the computation of $F_{i}(x)$ for all $i\in\cco 1,N\ccf$. In other words, with this naive construction, there is no numercial gain: $\na U(x) = \sum_{i=0}^N F_i(x)$ is computed at each time step. As a consequence, the algorithm is only useful if a relevant thinning procedure such 
as discussed in Section \ref{SubSection-FactorisationZZZ} is available, to avoid the systematic computation of $(F_i(x))_{i\in\cco 1,N\ccf}$ at each step. 
 Such a thinning relies on suitable bounds on the vector fields $(F_i)_{i\in\N}$, that depends on the form of $U$ and of the way $\na U$ is splitted. For this reason, in the rest of this section, we won't address this question in a general case, but rather  focus on a particular example. This will illustrate the fact that there exist cases where splitting forces between drift and jump mechanisms can decrease the total computational cost of the simulation.

\subsubsection{A motivating example}\label{Sec:LennardJones}

To fix ideas, consider a system of $M$ particles in the torus $(a\mathbb T)^3$ interacting through truncated a Lennard-Jones potential, in other words,  given some parameters $U_0,r,a>0$, the total energy of the configuration $x\in ((a\mathbb T)^3)^M$ is
\[U(x) \ = \  U_0 \sum_{i=1}^M \sum_{j\neq i} W \po |x_i-x_j| \pf \]
where  $W(h) = [(r/h)^{12}-(r/h)^6]\chi(h/a)$ with $\chi$ a $\mathcal C^2$ positive function with values in $[0,1]$ such that $\chi(s)=1$ for $s\leqslant 1/2$ and  $\chi(s)=0$ for $s\geqslant 1$ (in particular, a particle doesn't interact with its periodic image, or with several copies of the same other particle). Remark that strictly speaking this problem doesn't enter the framework considered above where, for the sake of  simplicity, $U$ was supposed smooth, but this is typically the kind of singular potentials met in molecular dynamics simulations.

For $i\in \cco 1,M\ccf$, denote $J_i$ the $3M\times 3$ matrix with zeros everywhere except $J_i(3(i-1)+1,1)=J_i(3(i-1)+2,2)=J_i(3(i-1)+3,3)=1$, so that $\na U = \sum_{i=1}^M J_i \na_{x_i} U$. For some $R\ll a$, we decompose $\na U = F_0 +\sum_{i=1}^M \sum_{j\neq i} J_i G_{i,j}$ with
\begin{eqnarray*}
F_0(x) &=&   U_0  \sum_{i=1}^M J_i \sum_{j\neq i} \frac{x_i-x_j}{|x_i-x_j|}W'(|x_i-x_j|) \chi(|x_i-x_j|/R) \\
G_{i,j}(x) & = &  U_0   \frac{x_i-x_j}{|x_i-x_j|}W'(|x_i-x_j|) \po 1 - \chi(|x_i-x_j|/R)\pf 
\end{eqnarray*}
 Then $F_0$ gathers the (singular) short-range forces, and the $G_{i,j}$'s the (bounded) long-range ones. Set
\begin{eqnarray}\label{Eq:splitLJ}
  F_{i,j} &=& J_i G_{i,j}\qquad \forall i,j\in\cco 1,M\ccf,\ i\neq j\,.
\end{eqnarray} 
In particular the number of jump mechanisms is $N=M(M-1)$. Remark that the jump rate associated with the vector field $F_{i,j}$ can be bounded as $(v\cdot J_i G_{i,j}(x))_+ \leqslant |v_i|U_0\sup\{|W'(s)|,\ s\geqslant R/2\}:=|v_i|C_R$.  As $R $ increases (keeping $R \ll a$), $C_R$ decays very fast   toward zero (more precisely, as $R^{-5}$).

Conversely, computing $F_0(x)$ involves computing, for each $i\in\cco 1,N\ccf$,  a sum over all the particles $j$ such that $\{|x_i-x_j|\leqslant R\}$. The numerical cost of this computation increases with $R$, but is small with respect to the computation of $\nabla U(x)$ as long as $R\ll a$.

\subsubsection{Thinning}
For the Lennard-Jones system with the decomposition $\na U = F_0 + \sum_{i=1}^M \sum_{j\neq i } F_{i,j}$ introduced above, the third step of the algorithm of Section~\ref{SubSectionStrang} can be achieved as follows:
\begin{itemize}
\item Set $W=\widetilde V_n$.
\item For all $i\in\cco 1,M\ccf$,
\item \qquad Draw $K_i$ according to a Poisson law with parameter $|W_i|C_RM\delta$.
\item \qquad For all $k\in \cco 1,K_i\ccf$,
\item \qquad \qquad Draw $j$ uniformly over $\cco 1,M\ccf$ and $U_{i,k}$ uniformly over $[0,1]$.
\item \qquad \qquad If $j\neq i$ and $U_{i,k} \leqslant \po W_i \cdot G_{i,j}(\widetilde X_n)\pf_+ / |W_i|C_R$, do
\[W_i \ \leftarrow \ R_{i,j}(x,W_i) :=  W_i - 2\po \frac{G_{i,j}(\widetilde X_n)\cdot W_i}{|G_{i,j}(\widetilde X_n)|^2}G_{i,j}(\widetilde X_n)\pf \,,\] 
\qquad \qquad else do nothing.
\item \qquad end for all $k$
\item end for all $i$.
\item Set $\widehat V_n = W$.
\end{itemize}
We call this the thinned algorithm. Remark that $|W_i|$ is unchanged by the reflection $R_{i,j}$,  which allows to draw $K_i$, the number of jump proposals for the $i^{th}$ particle, at the beginning of the loop (but in practice this is not important). More crucially, note also that each particle $i\in \cco 1,M\ccf$ can be treated in parallel, and there is no need for  time synchronization: indeed, two particles $i$ and $j$ only interact through $G_{i,j}(\widetilde X_n)$, and the positions $\widetilde X_n$ are fixed at this step, so that the velocity jumps of each particle doesn't affect the law of the other.

\subsubsection{Numerical efficiency}

Computing $\na U$ has a numerical cost of order $\mathcal O(M^2)$, and it is computed $T/\delta $ times if we sample a trajectory of a classical process (Langevin dynamics, etc.) in a time interval $[0,T]$ with a usual integrator with time-step $\delta $. Let us compare this with the cost of computing gradients in the thinned algorithm for the hybrid process.

First, denote $\mathcal N(R)$ the average number of particles that are at distance less than $R$ from a given particle. By using a Verlet list of neighbors, computing $F_0$ has an average cost of $\mathcal O\po M \mathcal N(R)\pf$.

Second, denote $D_T$ the number of times that $G_{i,j}$ has been evaluated for some $i,j\in\cco 1,M\ccf$ in a trajectory of length $T$ with time-step $\delta$. Then
\[D_T \ = \ \sum_{n=1}^{T/\delta} \sum_{i=1}^M S_{n,i}\]
where, conditionally to the $\widetilde V_{n,i}$'s the $S_{n,i}$'s are independent  Poisson random variables with parameter $|\widetilde V_{n,i}|C_RM\delta$. By the law of large numbers for ergodic Markov chains,
\[D_T \ \underset{T\rightarrow\infty}\simeq \ T C_RM^2  H\]
where $H$ is the average of $|w_i|$ with respect to the equilibrium distribution of the chain. When $\delta$ is small, the velocity distribution at equilibrium is close to a standard Gaussian one of dimension 3, so that $H\leqslant \sqrt{3} + \mathcal O(\delta)$.

Computing $G_{i,j}(x)$ being of order $\mathcal O(1)$, the total cost is  $\mathcal O(M^2 C_R T + M \mathcal N(R)T/\delta)$. As noted previously, as $R$ increases, $C_R$ decays but $\mathcal N(R)$ increases. In the regime where $\mathcal N(R) \ll M$,   the gain in the thinned algorithm, with respect to an integrator where the full gradient $\nabla U$ is computed at each step, is the ratio of $M^2 C_R T$ and $M^2T/\delta$, namely $C_R\delta $. In other words, the long-range forces are only evaluated at an average time-step $1/C_R$. Depending on the parameters, this can be a significant speed-up. Indeed, $\delta$ is typically constrained to be small in order to handle the singular behaviour of short-range forces while, for intermediate values of $R$, $C_R$ can already be small. We refer to \cite{MonmarcheBouncyChimie} for a practical case where $\delta C_R \simeq 10^{-6}$.

\subsubsection{Sampling rate: the mean-field regime}

The numerical efficiency discussed above is encouraging, but one should be careful when comparing two Markov processes for sampling. Indeed, suppose we are given two kinetic processes $(X_t,V_t)_{t\in[0,T]}$ and $(\tilde X_t,\tilde V_t)_{t\in[0,T]}$, both $\mu$-ergodic, such that simulating a trajectory of length $T$ for the first one is $10$ times faster than for the seconde one. This is useless if the first process explores the space $100$ times slower than the second one so that, to achieve a similar quality of sampling, much longer trajectories are required.

Such a situation may be feared for the hybrid drift/jump process introduced above in the Lennard-Jones model. Indeed, under the effect of many collisions, the process may show as $M\rightarrow \infty$  a diffusive behaviour, namely the jumps may average and give a velocity close to zero. In that case, the exploration of the space (hence the convergence toward equilibrium) would be very slow.  
 Let us (informally) check that it is not the case here, at least in the mean-field regime: in the following, we suppose that $U_0=1/M$ (in which case the constant $C_R$ is of order $1/M$).

Consider the Langevin dynamics with generator $L=A_1+A_2+A_4$ with $F_0 = \na U$. Suppose that the initial conditions $(x_i,v_i)_{i\in\cco 1,M\ccf}$ are i.i.d. with the $x_i$'s distributed according to some law $\nu$. Conditionally to $(x_1,v_1)$, by the law of large numbers, the force felt by the first particle at time $0$ satisfies
\[\na_{x_1} U(x) \ = \ \na_{x_1}  \frac1M \sum_{j=2}^M W(|x_i-x_j|) \ \underset{M\rightarrow+\infty}\longrightarrow \na_{x_1} \int_{(a\T)^3} W(|x_1-z|)\nu(\dd z)\,. \]
Classical propagation of chaos results show that, as $M\rightarrow +\infty$, the particles behaves approximately as $M$ independent processes on $(a\T)^3\times \R^3$ with generator
\[L_t f \ = \ v\cdot \na_x f- \na_x \po \int_{(a\T)^3} W(|x-z|) \nu_t (\dd  z) \pf  \cdot\na_v f - v\cdot \na_v f + \Delta_v f\,,\]
where $\nu_t$ is the law of the process at time $t$, solution of the non-linear equation $\partial_t (\nu_t f) = \nu_t L_t f$ for all nice $f$.

Similarly, for the hybrid jump/diffusion process obtained from the splitting of the forces introduced in Section~\ref{Sec:LennardJones}, conditionally to $(x_1,v_1)$, by the law of large numbers, denoting $g$ the function such that $G_{i,j}(x) = g(x_i-x_j)/M$ for $i,j\in\cco 1,M\ccf$, the jump rate at time 0 for the first particle  satisfies
\begin{eqnarray*}
\frac1M \sum_{j=2}^M\po v_1 \cdot G_{1,j}(x)\pf_+ & \underset{M\rightarrow\infty }\longrightarrow &  \int_{(a\T)^3}\po v_1 \cdot g(|x_1-z|)\pf_+\nu(\dd z)  :=  \lambda_{\nu}(x_1,v_1) \,,
\end{eqnarray*}
and a similar limit $Q_{\nu}$ is obtained for the jump kernel at time $0$.  From the propagation of chaos phenomenon, the system is expected to behave as $M\rightarrow \infty$ as $M$ independent non-linear processes on $(a\T)^3\times\R^3$ with non-homogeneous generators given by
\[L_t f \ = \ v\cdot \na_x f-\po \int_{(a\T)^3} \tilde g(x-z) \nu_t (\dd  z) \pf  \cdot\na_v f  +  \lambda_{\nu_t} \po   Q_{\nu_t} f - f\pf - v\cdot \na_v f + \Delta_v f\]
with a suitable function $\tilde g$, and $\nu_t$ the law of the process (see \cite{2018MonmarcheCouplage} and references within). The non-homogeneous Markov process with generator $L_t$ is not degenerated in the sense its velocity is not averaged to zero, which means that for large $M$ the system of interacting particles is not expected to have a diffusive behaviour, nor to converge to equilibrium with a rate that vanishes as $M\rightarrow +\infty$.


%

Of course the dynamics (Langevin versus hybrid) are different and thus they may converge toward equilibrium at different speed.  Nevertheless, the ratio of the convergence rate may be expected to be of order independent from $M$, since the mean-field limits as $M\rightarrow \infty$ are obtained with the same scaling (namely it was not necessary  to accelerate time). In other words, in order to achieve the same convergence as the Langevin process in a time $T$ (for a cost of order $\mathcal O(M^2T/\delta)$), we  may need to sample the hybrid process up to a time $T'>T$ (for a cost of order $\mathcal O(M\mathcal N(R) T'/\delta + M^2C_RT')$) but with a ratio $T'/T$ bounded below independently from $M$. For $R$ small enough (independent from $M$) so that $\mathcal N(R) < (T/T') M$, using that $C_R$ is $\mathcal O(1/M)$, we see that the hybrid  process appears as the most efficient for large $M$.

More generally, if $U_0$ is not of order $1/M$, the conclusion may be unclear, but the informal discussion above shows that there exist cases where the parameters of the problem are such that a suitable factorization can lower the numerical cost of the simulation at constant long-time convergence quality. We refer to \cite{MonmarcheBouncyChimie} for a practical case where the diffusion constant,  used as an indicator of the sampling rate, is of the same order for the two processes.

\subsubsection{An alternative splitting of the forces}

For the Lennard-Jones model, consider the splitting of the forces  $\na U = \sum_{i=0}^M F_i$ where $F_0$ is as above and  $F_i(x) = J_i \na_{x_i} U =  J_i\sum_{j\neq i} G_{i,j}$ for $i\in\cco 1,M\ccf$. Since $(v_i\cdot \sum_{j\neq i} G_{i,j})_+ \leqslant \sum_{j=1}^M (v_i\cdot G_{i,j})_+$, this alternative splitting reduces the total jump rate of the process, by comparison with the previous splitting. Increasing the jump rate is known to increase the diffusive behaviour and the
asymptotical variance for the Zig-Zag process, at least in dimension $1$ \cite{BierkensDuncan}. Since the present case is similar, we may expect the alternative splitting to yield a process that converges faster toward equilibrium than the initial one. Nevertheless, in the mean-field case, similarly to the previous section, this speed-up should be independent from $M$.

On the other hand, let us estimate the numerical cost (in term of computations of forces) of the alternative process.  We can bound $|F_i(x)|\leqslant MC_R$ (with the same $C_R$ as previously), and thus a thinning algorithm yields again an average of $MC_R T$ jump proposals per particle in a time period $[0,T]$. But now, at each jump proposal, $F_i(x)$ has to be computed, which has a cost $\mathcal O(M)$, so the total cost is $\mathcal O(M^3 C_R T)$.

This little computation indicates that, in the case of forces that come from pairwise interactions between particles, rather than considering one jump mechanisms per particle (as in the usual Zig-Zag process, and in the alternative splitting of this section), it may be better to split $\na_{x_i} U$ so that each jump mechanism is only associated to a single interaction $(i,j)$ (as in the initial splitting introduced in Section~\ref{Sec:LennardJones}).

\section*{Acknowledgements}

The author would like to thank Gabriel Stoltz and Louis Lagardère for fruitfull discussions about Strang splitting schemes, and acknowledges partial financial support from  the  ERC grant MSMATHS (European Union’s Seventh Framework Programme (FP/2007-2013)/ERC Grant Agreement number 614492) and the ERC grant EMC2, and from the French ANR grant EFI (Entropy, flows, inequalities,  ANR-17-CE40-0030).

\bibliographystyle{plain}
\bibliography{biblio_persistent}

\begin{thebibliography}{10}

\bibitem{Chaturvedi}
V.~Balakrishnan and S.~Chaturvedi.
\newblock Persistent diffusion on a line.
\newblock {\em Phys. A}, 148(3):581--596, 1988.

\bibitem{Krauth}
E.~P. {Bernard}, W.~{Krauth}, and D.~B. {Wilson}.
\newblock {Event-chain Monte Carlo algorithms for hard-sphere systems}.
\newblock {\em Phys. Rev. E}, 80(5):056704, November 2009.

\bibitem{BierkensDuncan}
J.~Bierkens and A.~Duncan.
\newblock Limit theorems for the zig-zag process.
\newblock {\em Adv. in Appl. Probab.}, 49(3):791--825, 2017.

\bibitem{BierkensFearnheadRoberts}
J.~Bierkens, P.~Fearnhead, and G.~Roberts.
\newblock The zig-zag process and super-efficient sampling for {B}ayesian
  analysis of big data.
\newblock {\em Ann. Statist.}, 47(3):1288--1320, 2019.

\bibitem{BierkensRoberts}
J.~Bierkens and G.~Roberts.
\newblock A piecewise deterministic scaling limit of lifted
  {M}etropolis-{H}astings in the {C}urie-{W}eiss model.
\newblock {\em Ann. Appl. Probab.}, 27(2):846--882, 2017.

\bibitem{BierkensRobertsZitt}
J.~{Bierkens}, G.~{Roberts}, and P.-A. {Zitt}.
\newblock Ergodicity of the zigzag process.
\newblock {\em Ann. Appl. Probab.}, 29(4):2266--2301, 2019.

\bibitem{BouRabee}
N.~Bou-Rabee.
\newblock Time integrators for molecular dynamics.
\newblock {\em Entropy}, 16:138--162, 2014.

\bibitem{Doucet}
A.~Bouchard-C\^{o}t\'{e}, S.~J. Vollmer, and A.~Doucet.
\newblock The bouncy particle sampler: a nonreversible rejection-free {M}arkov
  chain {M}onte {C}arlo method.
\newblock {\em J. Amer. Statist. Assoc.}, 113(522):855--867, 2018.

\bibitem{Bouin}
E.~{Bouin}, J.~{Dolbeault}, S.~{Mischler}, C.~{Mouhot}, and C.~{Schmeiser}.
\newblock {Hypocoercivity without confinement}.
\newblock {\em arXiv e-prints}, page arXiv:1708.06180, Aug 2017.

\bibitem{Renshaw}
A.Y. Chen and E.~Renshaw.
\newblock The general correlated random walk.
\newblock {\em J. Appl. Probab.}, 31(4):869--884, 1994.

\bibitem{GibsonCarter}
Gibson D.A. and E.A. Carter.
\newblock Time-reversible multiple time scale ab initio molecular dynamics.
\newblock {\em J. Phys. Chem.}, pages 13429--13434, 1993.

\bibitem{Diaconis2000}
P.~Diaconis, S.~Holmes, and R.~M. Neal.
\newblock Analysis of a nonreversible {M}arkov chain sampler.
\newblock {\em Ann. Appl. Probab.}, 10(3):726--752, 2000.

\bibitem{DiaconisMiclo}
P.~Diaconis and L.~Miclo.
\newblock On the spectral analysis of second-order {M}arkov chains.
\newblock {\em Ann. Fac. Sci. Toulouse Math. (6)}, 22(3):573--621, 2013.

\bibitem{DurmusGuillinMonmarche2018}
A.~{Durmus}, A.~{Guillin}, and P.~{Monmarch{\'e}}.
\newblock {Geometric ergodicity of the bouncy particle sampler}.
\newblock {\em ArXiv e-prints}, July 2018.

\bibitem{DurmusGuillinMonmarcheToolbox}
A.~{Durmus}, A.~{Guillin}, and P.~{Monmarch{\'e}}.
\newblock {Piecewise Deterministic Markov Processes and their invariant
  measure}.
\newblock {\em arXiv e-prints}, page arXiv:1807.05421, Jul 2018.

\bibitem{DurmusMoulines}
A.~Durmus and \'{E}. Moulines.
\newblock Nonasymptotic convergence analysis for the unadjusted {L}angevin
  algorithm.
\newblock {\em Ann. Appl. Probab.}, 27(3):1551--1587, 2017.

\bibitem{FontbonaGuerinMalrieu}
J.~Fontbona, H.~Gu\'{e}rin, and F.~Malrieu.
\newblock Long time behavior of telegraph processes under convex potentials.
\newblock {\em Stochastic Process. Appl.}, 126(10):3077--3101, 2016.

\bibitem{GlynnMeyn}
Peter~W. Glynn and Sean~P. Meyn.
\newblock A {L}iapounov bound for solutions of the {P}oisson equation.
\newblock {\em Ann. Probab.}, 24(2):916--931, 1996.

\bibitem{Goldstein}
S.~Goldstein.
\newblock On diffusion by discontinuous movements, and on the telegraph
  equation.
\newblock {\em Quart. J. Mech. Appl. Math.}, 4:129--156, 1951.

\bibitem{GruberThese}
U.~Gruber.
\newblock Convergence of binomial large investor models and general correlated
  random walks.
\newblock {\em PhD thesis, Technical University of Berlin}, 2004.

\bibitem{Gruber}
U.~Gruber and M.~Schweizer.
\newblock A diffusion limit for generalized correlated random walks.
\newblock {\em J. Appl. Probab.}, 43(1):60--73, 2006.

\bibitem{Hadeler}
K.~P. Hadeler.
\newblock Travelling fronts for correlated random walks.
\newblock {\em Canad. Appl. Math. Quart.}, 2(1):27--43, 1994.

\bibitem{HairerLubichWanner}
E.~Hairer, C.~Lubich, and G.~Wanner.
\newblock Geometric numerical integration illustrated by the störmer–verlet
  method.
\newblock {\em Acta Numerica}, 12:399–450, 2003.

\bibitem{HairerMattingly2008}
M.~Hairer and J.~C. Mattingly.
\newblock Yet another look at {H}arris' ergodic theorem for {M}arkov chains.
\newblock In {\em Seminar on {S}tochastic {A}nalysis, {R}andom {F}ields and
  {A}pplications {VI}}, volume~63 of {\em Progr. Probab.}, pages 109--117.
  Birkh\"auser/Springer Basel AG, Basel, 2011.

\bibitem{HerrmannVallois}
S.~Herrmann and P.~Vallois.
\newblock From persistent random walk to the telegraph noise.
\newblock {\em Stoch. Dyn.}, 10(2):161--196, 2010.

\bibitem{Kac}
M.~Kac.
\newblock A stochastic model related to the telegrapher's equation.
\newblock {\em Rocky Mountain J. Math.}, 4:497--509, 1974.
\newblock Reprinting of an article published in 1956, Papers arising from a
  Conference on Stochastic Differential Equations (Univ. Alberta, Edmonton,
  Alta., 1972).

\bibitem{Kallenberg}
O.~Kallenberg.
\newblock {\em Foundations of modern probability}.
\newblock Probability and its Applications (New York). Springer-Verlag, New
  York, 1997.

\bibitem{Turitsyn}
Turitsyn K.S., Chertkov M., and Vucelja M.
\newblock Irreversible monte carlo algorithms for efficient sampling.
\newblock {\em Physica D: Nonlinear Phenomena}, 240(4):410 -- 414, 2011.

\bibitem{Leimkuhler}
B.~Leimkuhler and C.~Matthews.
\newblock Rational construction of stochastic numerical methods for molecular
  sampling.
\newblock {\em Appl. Math. Res. Express. AMRX}, (1):34--56, 2013.

\bibitem{LeimkuhlerMatthewsStoltz}
B.~Leimkuhler, C.~Matthews, and G.~Stoltz.
\newblock The computation of averages from equilibrium and nonequilibrium
  {L}angevin molecular dynamics.
\newblock {\em IMA J. Numer. Anal.}, 36(1):13--79, 2016.

\bibitem{LelievreStoltz}
T.~Leli\`evre and G.~Stoltz.
\newblock Partial differential equations and stochastic methods in molecular
  dynamics.
\newblock {\em Acta Numer.}, 25:681--880, 2016.

\bibitem{Thieullen2016}
V.~{Lemaire}, M.~{Thieullen}, and N.~{Thomas}.
\newblock Exact simulation of the jump times of a class of piecewise
  deterministic {M}arkov processes.
\newblock {\em J. Sci. Comput.}, 75(3):1776--1807, 2018.

\bibitem{Maigret}
N.~Maigret.
\newblock Th\'{e}or\`eme de limite centrale fonctionnel pour une cha\^{i}ne de
  {M}arkov r\'{e}currente au sens de {H}arris et positive.
\newblock {\em Ann. Inst. H. Poincar\'{e} Sect. B (N.S.)}, 14(4):425--440
  (1979), 1978.

\bibitem{Brill}
M.~Mandelbaum, M.~Hlynka, and P.~H. Brill.
\newblock Nonhomogeneous geometric distributions with relations to birth and
  death processes.
\newblock {\em TOP}, 15(2):281--296, 2007.

\bibitem{MattinglyStuartretyakov}
J.~C. Mattingly, A.~M. Stuart, and M.~V. Tretyakov.
\newblock Convergence of numerical time-averaging and stationary measures via
  {P}oisson equations.
\newblock {\em SIAM J. Numer. Anal.}, 48(2):552--577, 2010.

\bibitem{Michelforward}
M.~{Michel}, A.~{Durmus}, and S.~{S{\'e}n{\'e}cal}.
\newblock {Forward Event-Chain Monte Carlo: Fast sampling by randomness control
  in irreversible Markov chains}.
\newblock {\em arXiv e-prints}, page arXiv:1702.08397, Feb 2017.

\bibitem{Michelclock}
X.~{Michel}, M.and~{Tan} and Y.~{Deng}.
\newblock {Clock Monte Carlo methods}.
\newblock {\em arXiv e-prints}, page arXiv:1706.10261, Jun 2017.

\bibitem{Monmarche2013}
L.~Miclo and P.~Monmarch\'{e}.
\newblock \'{E}tude spectrale minutieuse de processus moins ind\'{e}cis que les
  autres.
\newblock In {\em S\'{e}minaire de {P}robabilit\'{e}s {XLV}}, volume 2078 of
  {\em Lecture Notes in Math.}, pages 459--481. Springer, Cham, 2013.

\bibitem{MonmarcheRTP}
P.~Monmarch{\'e}.
\newblock Piecewise deterministic simulated annealing.
\newblock {\em ALEA Lat. Am. J. Probab. Math. Stat.}, 13(1):357--398, 2016.

\bibitem{Monmarche2017}
P.~Monmarch\'{e}.
\newblock Weakly self-interacting velocity jump processes for bacterial
  chemotaxis and adaptive algorithms.
\newblock {\em Markov Process. Related Fields}, 23(4):609--659, 2017.

\bibitem{2018MonmarcheCouplage}
P.~{Monmarch{\'e}}.
\newblock {Elementary coupling approach for non-linear perturbation of Markov
  processes with mean-field jump mechanims and related problems}.
\newblock {\em arXiv e-prints}, page arXiv:1809.10953, Sep 2018.

\bibitem{MonmarcheBouncyChimie}
P.~{Monmarch{\'e}}, J.~{Weisman}, and J.-P. {Lagard{\`e}re}, L.and~{Piquemal}.
\newblock {Velocity jump processes : an alternative to multi-timestep methods
  for faster and accurate molecular dynamics simulations}.
\newblock {\em arXiv e-prints}, page arXiv:2002.07109, Feb 2020.

\bibitem{PetersdeWith}
E.~A. J.~F. Peters and G.~de~With.
\newblock Rejection-free monte carlo sampling for general potentials.
\newblock {\em Phys. Rev. E 85, 026703}, 2012.

\bibitem{RicciCiccotti}
A.~Ricci and G.~Ciccotti.
\newblock Algorithms for brownian dynamics.
\newblock {\em Molecular Physics}, 101(12):1927--1931, 2003.

\bibitem{Rosetto}
V.~Rossetto.
\newblock The one-dimensional asymmetric persistent random walk.
\newblock {\em J. Stat. Mech.}, 2018.

\bibitem{TalayTubaro}
D.~Talay and L.~Tubaro.
\newblock Expansion of the global error for numerical schemes solving
  stochastic differential equations.
\newblock {\em Stochastic Anal. Appl.}, 8(4):483--509 (1991), 1990.

\bibitem{TuckermanRossiBerne}
M.E. Tuckerman, B.J Berne, and Rossi.
\newblock Molecullar dynamics algorithm for multiple time scales: Systems with
  disparate masses.
\newblock {\em J. Chem. Phys.}, 1991.

\bibitem{DoucetPDMCMC}
P.~{Vanetti}, A.~{Bouchard-C{\^o}t{\'e}}, G.~{Deligiannidis}, and A.~{Doucet}.
\newblock {Piecewise Deterministic Markov Chain Monte Carlo}.
\newblock {\em ArXiv e-prints}, July 2017.

\bibitem{RobertWu}
C.~{Wu} and C.~P. {Robert}.
\newblock {Generalized Bouncy Particle Sampler}.
\newblock {\em arXiv e-prints}, page arXiv:1706.04781, Jun 2017.

\bibitem{Zauderer}
E.~Zauderer.
\newblock Correlated random walks, hyperbolic systems and fokker-planck
  equations.
\newblock {\em Mathematical and Computer Modelling}, 17(10):43 -- 47, 1993.

\end{thebibliography}

\end{document}